%% file: groupKernels.tex
\documentclass[preprint,12pt]{article}   

\RequirePackage[OT1]{fontenc}
\RequirePackage{amsthm,amsmath}
\usepackage{amsfonts}
\usepackage{bbm}
\usepackage{stmaryrd}
\RequirePackage{natbib}
\RequirePackage[colorlinks,citecolor=blue,urlcolor=blue]{hyperref}

\usepackage{bbm}
\usepackage{xcolor}
\usepackage{graphicx}
\usepackage{caption}
\usepackage{subcaption}
\usepackage{bm}

\newcommand{\mb}[1]{\mathbf{#1}}

\newcommand{\Var}{\mathrm{var}}
\newcommand{\Cov}{\mathrm{cov}}

\newcommand{\kcont}{k_{\textrm{cont}}}
\newcommand{\kcat}{k_{\textrm{cat}}}
\newcommand{\one}[1]{\mathbf{1}_{#1}}

\newcommand{\compound}[2]{\mb{\Gamma}^{\mathrm{CS}}_{#1} \left(#2\right)}

\newcommand{\groupEffect}{\mu}
\newcommand{\levelEffect}{\lambda}
\newcommand{\varIndiv}{v_\levelEffect} 
\newcommand{\varCentre}{v_\groupEffect} 
\newcommand{\varCentreGroup}{\mb{B}^\star}
\newcommand{\varCentreGroupCoef}[1]{B_{#1}^\star}
\newcommand{\varIndivG}[1]{\mb{W}_{#1}^\star}
\newcommand{\varIndivGbar}[1]{\overline{W_{#1}^\star}}
\newcommand{\varZeroAverageRed}[1]{\mb{M}_{#1}}
\newcommand{\constrastMat}[1]{\mb{A}_{#1}}
\newcommand{\J}[1]{\mb{J}_{#1}}
\newcommand{\I}[1]{\mb{I}_{#1}}
\newcommand{\nbGroup}{G}
\newcommand{\domain}{\mathcal{D}}
\newcommand{\group}{\mathcal{G}}
\newcommand{\within}{W}
\newcommand{\between}{\mb{B}}
\newcommand{\level}{L}
\newcommand{\levelIndex}{\ell}
\newcommand{\chol}{\mb{L}}
\newcommand{\matCovCat}{T}
\newcommand{\responseAnova}{\eta}

\renewcommand{\ell}{u}

\usepackage{authblk}   

\title{Group kernels for Gaussian process metamodels with categorical inputs}

\author[1]{O. Roustant}
\author[1]{E. Padonou}
\author[2]{Y. Deville}
\author[3]{A. Cl\'ement}
\author[4]{G. Perrin}
\author[4]{J. Giorla}
\author[5]{H. Wynn}
\affil[1]{{\small Mines Saint-\'{E}tienne, Univ. Clermont Auvergne, CNRS, UMR 6158 LIMOS, F--42023 Saint-\'{E}tienne, France}}
\affil[2]{{\small AlpeStat, Chamb\'ery, France}}
\affil[3]{{\small CEA/DAM/VA, F--21120, Is-sur-Tille, France}}
\affil[4]{{\small CEA/DAM/DIF, F--91297, Arpajon, France}}
\affil[5]{{\small London School of Economics, England}}

\date{} 

\begin{document}

\maketitle

\begin{abstract}
Gaussian processes (GP) are widely used as a metamodel for emulating time-consuming computer codes.
We focus on problems involving categorical inputs, with a potentially large number $\level$ of levels (typically several tens),
partitioned in $\nbGroup \ll \level$ groups of various sizes. 
Parsimonious covariance functions, or kernels, can then be defined by block covariance matrices $\mb{\matCovCat}$
with constant covariances between pairs of blocks and within blocks.
We study the positive definiteness of such matrices to encourage their practical use.
The hierarchical group/level structure, equivalent to a nested Bayesian linear model,
provides a parameterization of valid block matrices $\mb{\matCovCat}$. 
The same model can then be used when the assumption within blocks is relaxed,
giving a flexible parametric family of valid covariance matrices with constant covariances between pairs of blocks.
The positive definiteness of $\mb{\matCovCat}$ 
is equivalent to the positive definiteness of a smaller matrix of size $\nbGroup$,
obtained by averaging each block.
The model is applied to a problem in nuclear waste analysis,
where one of the categorical inputs is atomic number, 
which has more than 90 levels.
\end{abstract}

\section{Introduction} \label{sec:intro}
\input{introduction.tex}

\section{Background and notations} \label{sec:background}
\subsection{GPs with continuous and categorical variables}
\input{formulation.tex}

\subsection{1-dimensional kernels for categorical variables} \label{sec:overview}
\input{background.tex}

\section{Generalized compound symmetry block covariance matrices} \label{sec:groupCovariance}
\input{blockCovariance.tex}

\section{Example} \label{sec:examples}
\input{example.tex}

\section{Application in nuclear engineering} \label{sec:application}
\input{application.tex}

\section{Conclusion} \label{sec:conclusion}
\input{conclusion.tex}

\section*{Software and acknowledgements}
Implementations have been done with the R packages \texttt{mixgp} and \texttt{kergp} \citep{kergp}. 
Illustrations use \texttt{ggplot2} \citep{ggplot2} and \texttt{corrplot} \citep{corrplot}.\\
This research was conducted within the frame of the Chair in Applied Mathematics OQUAIDO, gathering partners in technological research (BRGM, CEA, IFPEN, IRSN, Safran, Storengy) and academia (CNRS, Ecole Centrale de Lyon, Mines Saint-Etienne, University of Grenoble, University of Nice, University of Toulouse) around advanced methods for Computer Experiments. The authors thank the participants for fruitful discussions.
They are also grateful to the Isaac Newton Institute for Mathematical Sciences, Cambridge, 
for support and hospitality during the programme UNQ when work on this paper was undertaken
(EPSRC grant number EP/K032208/1).

\section*{Appendix}
\input{appendix.tex}

\bibliographystyle{abbrvnat}

\bibliography{groupCovariance} 

\end{document}

%% file: introduction.tex

This research is motivated by the analysis of a time-consuming computer code in nuclear engineering,
depending on both continuous and categorical inputs, one of them having more than 90 levels.
The final motivation is an inversion problem.
However, due to the heavy computational cost, a direct usage of the simulator is hardly possible.
A realistic approach is to use a statistical emulator or metamodel. 
Thus, as a first step, we investigate the metamodelling of such computer code.
More precisely, we consider Gaussian process (GP) regression models, also called kriging models 
(\citep{Sacks_1989}, \citep{rasmussen2005}),
which have been successfully used in sequential metamodel-based strategies for uncertainty quantification
(see e.g. \citep{Chevalier_2014}).\\ 

Whereas there is a flourishing literature on GP regression, 
the part concerned with categorical inputs remains quite limited.
We refer to \citep{zhang2015} for a review.
As for continuous inputs, covariance functions or \textit{kernels}
are usually built by combination of 1-dimensional ones, 
most often by multiplication or, more rarely, by addition \citep{Deng2017}.
The question then comes down to constructing a valid kernel on a finite set,
which is a positive semidefinite matrix. 
Some effort has been spent on parameterization of general covariance matrices
\citep{Pinheiro1996} and parsimonious parameterizations of smaller classes \citep{pinheiro2009}.
Some block forms have also been proposed \citep{qian2007}, in order to deal with a potential large number of levels.
However, their validity (in terms of positive definiteness) was not investigated.
Furthermore, to the best of our knowledge, 
applications in GP regression are limited to categorical inputs with very few levels, typically less than 5.\\

Guided by the application, we investigate more deeply the so-called \textit{group kernels} cited in \citet{qian2007},
defined by block covariance matrices $\mb{\matCovCat}$ with constant covariances between pairs of blocks and within blocks.
We exploit the hierarchy group/level by revisiting a nested Bayesian linear model 
where the response term is a sum of a group effect and a level effect.
The level effects are assumed to sum to zero, which allows recovering negative within-group correlations.
This model leads to a parameterization of $\mb{\matCovCat}$ which is automatically positive definite.
Interestingly, the assumption on within blocks can be relaxed, 
and we obtain a parameterization of a wider class of valid group kernels.
The positive definiteness condition of $\mb{\matCovCat}$ is also explicited:
it is equivalent to the positive definiteness of the smaller covariance matrix 
obtained by replacing each block by its average.\\

As mentioned above, this work has some connections with Bayesian linear models 
as well as linear mixed effect models (see e.g. \citet{Lindley_Smith_1972}, \citet{smith1973}) 
 in a hierarchical view.
Other related works concern hierarchical GPs with a tree structure.
For instance, particular forms of group kernels are obtained in multiresolution GP models 
(\citep{multiresolutionGP_Fox_Dunson}, \citep{Park_Shoi_2010}). 
Such models usually assume that children are conditionally independent on the mother.
This is not the case in our model, due to the condition that the level effects sum to zero.\\

The paper is structured as follows. 
Section \ref{sec:background} gives some background on GP regression with mixed categorical and continuous inputs. 
Section \ref{sec:groupCovariance} presents new findings on group kernels.
Section \ref{sec:examples} illustrates on synthetic examples.
Section \ref{sec:application} is devoted to the application which motivated this work.
Section \ref{sec:conclusion} gives some conclusions and perspectives for future research.

%% file: formulation.tex

We consider a set of $I$ continuous variables $\mb{x}_1, \dots, \mb{x}_I$ defined on a hypercubic domain $\Delta$,
and a set of $J$ categorical variables $u_1, \dots, u_J$ with $\level_1, \dots, \level_J$ levels.
Without loss of generality, we assume that $\Delta = [0,1]^I$ and that, for each $j=1, \dots, J$,
the levels of $u_j$ are numbered $1, 2, \dots, \level_j$.
We denote $\mb{\mb{x}} = (\mb{x}_1, \dots, \mb{x}_I)$, $\mb{u} = (u_1, \dots, u_J)$, and $\mb{\mb{w}} = (\mb{\mb{x}},\mb{u})$.\\

We consider GP regression models defined on the product space
\begin{equation*}
\label{eq:space}
\domain = [0, 1]^I \times \prod_{j=1}^J \lbrace 1, \dots, \level_j \rbrace,
\end{equation*}
and written as:  
\begin{equation}
\label{eq:GPR2}
y_i = \mu (\mb{w}^{(i)}) + Z({\mb{w}}^{(i)}) + \epsilon_i, \qquad i = 1, \dots, N.
\end{equation}
where $\mu$, $Z$ and $\epsilon$ are respectively the trend, the GP part and a noise term. 
There exist a wide variety of trend functions, as in linear models. 
Our main focus here is on the centered GP $Z(\mb{w})$, characterized by its kernel 
$$k: (\mb{\mb{w}}, \mb{\mb{w}'}) \mapsto \Cov \left( Z(\mb{\mb{w}}),  Z(\mb{\mb{w}'}) \right).$$
Kernels on $\domain$ can be obtained by combining kernels on $[0, 1]^I$ 
and kernels on $\prod_{j=1}^J \lbrace 1, \dots, \level_j \rbrace $. 
Standard valid combinations are the product, sum or ANOVA.
Thus if $\kcont$ denotes a kernel for the continuous variables $\mb{x}$,
$\kcat$ a kernel for the categorical ones $\mb{u}$, 
examples of valid kernels for $\mb{w} = (\mb{x},\mb{u})$ are written:
\begin{eqnarray*}
(\textrm{Product}) \qquad k(\mb{w},\mb{w}') &=& \kcont(\mb{x},\mb{x}')\kcat(\mb{u},\mb{u}') \\
(\textrm{Sum}) \qquad k(\mb{w},\mb{w}') &=& \kcont(\mb{x},\mb{x}') + \kcat(\mb{u},\mb{u}') \\
(\textrm{ANOVA}) \qquad k(\mb{w},\mb{w}') &=&  (1+\kcont(\mb{x},\mb{x}'))(1+\kcat(\mb{u},\mb{u}'))
\end{eqnarray*}
For consiseness, we will denote by $\ast$ one of the operations: sum, product or ANOVA.
The three formula above can then be summarized by:
\begin{equation} \label{eq:contCatComb}
k(\mb{w},\mb{w}') = \kcont(\mb{x},\mb{x}') \ast \kcat(\mb{u},\mb{u}')
\end{equation}

Then, in turn, $\kcont$ and $\kcat$ can be defined by applying these operations to $1$-dimensional kernels.
For continuous variables, famous $1$-dimensional kernels include squared exponential or Mat\'ern  \citep{rasmussen2005}. 
We denote by $\kcont^i(x_i, x'_i)$ such kernels ($i = 1, \dots, I$).
For a categorical variable, notice that, as a positive semidefinite function on a finite space, 
a kernel is a positive semidefinite matrix. 
We denote by $\mb{\matCovCat}_j$ the matrix of size $\level_j$ corresponding to kernels for $u_j$ ($j = 1, \dots, J$).
Thus, examples of expressions for $\kcont$ and $\kcat$ are written:
\begin{eqnarray} 
\kcont (\mb{x},\mb{x}') &=& \kcont^1(x_1, x'_1) \ast \dots \ast \kcont^I(x_I, x'_I) \label{eq:kContComb} \\
\kcat (\mb{u},\mb{u}') &=& \left[ \matCovCat_1 \right]_{u_1, u'_1} \ast \dots \ast \left[ \matCovCat_J \right]_{u_J, u'_J}
\label{eq:kCatComb}
\end{eqnarray}

The formulation given by Equations~(\ref{eq:contCatComb}), (\ref{eq:kContComb}), (\ref{eq:kCatComb})
is not the most general one, since kernels are not always obtained by combining 1-dimensional ones.
Nevertheless, it encompasses the GP models used in the literature of computer experiments with categorical inputs.
It generalizes the tensor-product kernels, very often used, and the sum used recently by \citep{Deng2017} on the categorical part.
It also contains the heteroscedastic case, since the matrices $\mb{\matCovCat}_j$ are not assumed to have a constant diagonal, 
contrarily to most existing works \citep{zhang2015}. 
This will be useful in the application of Section~\ref{sec:application}, 
where the variance of the material is level dependent.

\newtheorem{remIdentify}{Remark}
\begin{remIdentify} \label{rk:identifiability}
Combining kernels needs some care to obtain identifiable models.
For instance, the product of kernels $k_1, k_2$ with $k_i(x_i,x'_i)=\sigma_i^2 e^{-\vert x_i - x'_i \vert}$ ($i = 1, 2$),
is a kernel depending on only one variance parameter $\sigma^2 := \sigma_1^2 \sigma_2^2$.
The GP model is identifiable for this new parameter, but not for the initial parameters $\sigma_1^2,\sigma_2^2$.
\end{remIdentify}

%% file: background.tex

We consider here a single categorical variable $u$ with levels $1, \dots, \level$.
We recall that a kernel for $u$ is then a $\level$ by $\level$ positive semidefinite matrix $\mb{\matCovCat}$. 
 
\subsubsection{Kernels for ordinal variables} \label{sec:ordinalKernel}
A categorical variable with ordered levels is called ordinal. 
In this case, the levels can be viewed as a discretization of a continuous variable.
Thus a GP $Y$ on $\{1, \dots, \level\}$ can be obtained from a 1-dimensional GP $Z$ 
on the interval $[0,1]$ by using a non-decreasing transformation $F$ (also called warping):
$$ Y(u) = Z(F(u)).$$
Consequently, the covariance matrix $\mb{\matCovCat}$ can be written:
\begin{equation} \label{eq:ordinalKernel}
\matCovCat_{\levelIndex, \levelIndex'} = k_Z(F(\levelIndex), F(\levelIndex')), \quad \levelIndex, \levelIndex' = 1, \dots, \level.
\end{equation} 
When $k_Z(x, x')$ depends on the distance $\vert x - x' \vert$, then
$\matCovCat_{u, u'}$ 
depends on the distance between the levels $\levelIndex, \levelIndex'$, distorted by $F$.\\
In the general case, $F$ is piecewise-linear and defined by $\level-1$ parameters. 
However, a parsimonious parameterization may be preferred,  
based on the cdf of a flexible probability distribution such as the Normal or the Beta.
We refer to \citep{mccullagh1980} for examples in regression 
and to \citep{qian2007} for illustrations in computer experiments.\\
There is also some flexibility in the choice of the continuous kernel $k_Z$.
The standard Squared-Exponential or Mat\'ern kernels are admissible, but induce positive correlation between levels.
In order to allow negative correlations, one may choose, for instance, the cosine correlation kernel on $[0, \alpha)$:
\begin{equation} \label{eq:cosKernel}
k_Z(x, x') = \cos( x - x' ) 
\end{equation}
where $\alpha \in (0, \pi]$ is a fixed parameter tuning the minimal correlation value.
Indeed, (\ref{eq:cosKernel}) defines a decreasing function of $\vert x - x' \vert$ 
from $[0, \alpha]$ to $[\cos(\alpha), 1]$.
It is a valid covariance function obtained by choosing $\mu$ as a Dirac non-negative measure 
in Bochner theorem for real-valued stationary kernels:
$ k_Z(x, x') = \int \cos(\omega (x - x')) d\mu(\omega).$

\subsubsection{Kernels for nominal variables}
For simplicity we present here the homoscedastic case, 
i.e. when $\mb{\matCovCat}$ has a constant diagonal. 
It is immediately extended to situations where the variance depends on the level,
by considering the correlation matrix.

\paragraph{General parametric covariance matrices.} \label{sec:parametricCovMat} 
There are several parameterizations of positive-definite matrices based on the spectral and Choleky decompositions.
The spectral decomposition of $\mb{\matCovCat}$ is written
\begin{equation}
\mb{\matCovCat} = \mb{P} \mb{D} \mb{P}^\top
\end{equation} 
where $\mb{D}$ is diagonal and $\mb{P}$ orthogonal. 
Standard parameterizations of $\mb{P}$ involve the Cayley transform, 
Eulerian angles, Householder transformations or Givens rotations, as detailed in \citep{Khuri1989} and \citep{Ron2015}. 
Another general parameterization of $\mb{\matCovCat}$ is provided by the Cholesky decomposition:
\begin{equation} 
\mb{\matCovCat} = \mb{\chol} \mb{\chol}^\top,
\end{equation} 
where $\mb{\chol}$ is lower triangular. 
When the variance $\matCovCat_{\ell, \ell}$ does not depend on the level $\ell$, 
the columns of $\mb{\chol}$ have the same norm and represent points on a sphere in $\mathbb{R}^\level$. 
A spherical parameterization of $\mb{\chol}$ is then possible with one variance term and $\level(\level-1)/2$ angles, 
representing correlations between levels \cite[see e.g.][]{Pinheiro1996}. 
  
\paragraph{Parsimonious parameterizations.} \label{sec:parsimonious}
The general parametrizations of $\mb{\matCovCat}$ described above require $O(\level^2)$ parameters. 
More parsimonious ones can be used, up to additional model assumptions.
Among the simplest forms, the compound symmetry (CS) - often called exchangeable - covariance matrix
assumes a common correlation for all levels \citep[see e.g.][]{pinheiro2009}.
The CS matrix with variance $v$ and covariance $c$ is defined by:
\begin{equation}
\label{eq:compsym}
\matCovCat_{\levelIndex, \levelIndex'} = \left\lbrace \begin{matrix}
    v  & \text{ if } \levelIndex = \levelIndex'\\
	c  & \text{ if } \levelIndex \neq \levelIndex' \\
    \end{matrix} \right. 
    , \quad c/v \in \left(  -1/(\level-1), 1\right).
\end{equation}
This generalizes the kernel obtained by substituting the Gower distance $d$ \citep{gower1982} 
into the exponential kernel, corresponding to $c/v = e^{-d^2} >0$.\\
The CS covariance matrix treats equally all pairs of levels, which is an important limitation, especially when $\level \gg 1$. 
More flexibility is obtained by considering groups of levels. 
Assume that the $\level$ levels of $u$ are partitioned in $\nbGroup$ groups $\group_1, \dots, \group_\nbGroup$ 
and denote by $g(\levelIndex)$ the group number corresponding to a level $\levelIndex$.
Then a desired parameterization of $\mb{\matCovCat}$ is given by the block matrix (see e.g. \citet{qian2007}):
\begin{equation}
\label{eq:groupCovMatrix}
\matCovCat_{\levelIndex, \levelIndex'} = \left\lbrace \begin{matrix}
    v & \text{ if }  \levelIndex = \levelIndex'\\
	c_{g(\levelIndex), g(\levelIndex')} & \text{ if } \levelIndex \neq \levelIndex'\\
	\end{matrix} \right.
\end{equation}
where for all $i,j \in \{ 1, \dots, \nbGroup \}$, the terms $c_{i,i}/v$ are within-group correlations, 
and $c_{i,j}/v$ ($i \neq j$) are between-group correlations.
Notice that additional conditions on the $c_{i,j}$'s are necessary to ensure that $\mb{\matCovCat}$ is a valid covariance matrix, 
which is developed in the next section.

%% file: blockCovariance.tex

We consider the framework of Section~\ref{sec:parsimonious} where 
$u$ denotes a categorical variable whose levels are partitioned 
in $\nbGroup$ groups $\group_1, \dots, \group_\nbGroup$ of various sizes $n_1, \dots, n_\nbGroup$.
Without loss of generality, we 
assume that $\group_1=\{1, \dots, n_1\}, \group_2 = \{n_1+1, \dots, n_1+n_2\}, \dots$. 
We are interested in parsimonious parameterizations of the covariance matrix $\mb{\matCovCat}$,
written in block form:
\begin{equation}
\mb{\matCovCat} = \begin{pmatrix}  \label{eq:blockCorMat}
    \mb{\within}_1  & \between_{1,2} & \cdots & \between_{1,\nbGroup} \\
    \between_{2,1} & \within_2 & \ddots & \vdots \\  
    \vdots & \ddots & \ddots & \between_{\nbGroup-1, \nbGroup} \\
	\between_{\nbGroup,1} & \cdots & \between_{\nbGroup, \nbGroup-1} & \mb{\within}_\nbGroup  \\
    \end{pmatrix}
\end{equation} 
where the diagonal blocks $\mb{\within}_g$ contain within-group covariances, 
and the off-diagonal blocks $\between_{g,g'}$ are constant matrices containing between-group covariances.
We denote:
$$ \between_{g,g'} = c_{g,g'} \J{n_g, n_g'}, \qquad g \neq g' \in \{ 1, \dots, \nbGroup \}$$
where $\J{s,t}$ is the $s$ by $t$ matrix of ones. 
This means that the between-group covariances only depends on groups (and not on levels).\\

Although block matrices of the form (\ref{eq:blockCorMat}) may be covariance matrices, 
they are not positive semidefinite in general. 
A necessary condition is that all diagonal blocks $\mb{\within}_g$ are positive semidefinite. 
But it is not sufficient. 
In order to provide a full characterization, we will ask a little more, 
namely that they remain positive semidefinite when removing the mean:
\begin{equation} \label{eq:GCSdef}
\mb{\within}_g - \overline{\within_g} \J{n_g} \quad  \textrm{is positive semidefinite, for all} \quad  g=1, \dots, \nbGroup
\end{equation} 
where $\J{n_g}$ is matrix of ones of size $n_g$ and $\overline{\within_g}$ is the average of $\mb{\within}_g$ coefficients.
This condition will appear naturally in Subsection \ref{sec:hierarchicalGP}.
Notice that valid CS covariance matrices satisfy it.
Indeed, if $\mb{W}$ is a positive semidefinite matrix with variance $v$ and covariance $c$, then 
$\mb{W} - \overline{W} \J{n} = (v-c) \mb{P}$ 
where $\mb{P} = \I{n} - n^{-1} \J{n}$ verifies $\mb{P} = \mb{P} \mb{P}^\top$, which is positive semidefinite.  
For this reason, we will call matrices with
\emph{Generalized Compound Symmetry (GCS)}, 
block matrices of the form (\ref{eq:blockCorMat}) verifying (\ref{eq:GCSdef}).
In particular, the class of GCS block matrices contains block matrices of the form (\ref{eq:groupCovMatrix}). \\

The rest of the section is organized as follows.
Subsection \ref{sec:oneWayANOVA} shows how valid CS covariance matrices can be parameterized by a Gaussian model.
The correspondence is ensured, thanks to a centering condition on the effects of the levels.
Subsection \ref{sec:centeredCovMat} gives material on centered covariance matrices.
Subsection \ref{sec:hierarchicalGP} contains the main results. 
It extends the model of Subsection \ref{sec:oneWayANOVA} to a GCS block matrix. 
This gives a proper characterization of positive semidefinite GCS block matrices,
as well as a parameterization which automatically fulfills the positive semidefinite conditions. 
Subsection \ref{sec:relatedWorks} indicates connections with related works.
Finally, in Subsection \ref{sec:summary}, we collect together the details of our parameterization, 
for ease of reference.\\

\subsection{A Gaussian model for CS covariance matrices} \label{sec:oneWayANOVA}
We first focus on the case of a CS matrix. 
The following additional notations will be used: 
for a given integer $\level \geq 1$, 
$\I{\level}$ is the identity matrix of size $\level$, 
$\one{\level}$ is the vector of ones of size $L$.
We denote by
\begin{equation} \label{eq:CSmat}
\compound{\level}{v,c} = (v-c) \I{\level} + c \J{\level}
\end{equation}
the CS matrix with a common variance term $v$ and a common covariance term $c$.
It is well-known that $\compound{\level}{v,c}$ is positive definite if and only if 
\begin{equation} \label{eq:CSpdCondition}
-(\level-1)^{-1}v < c < v.
\end{equation}
For instance, one can check that the eigenvalues of $\compound{\level}{v,c}$
are ${v + (\level - 1)c}$ with multiplicity 1 (eigenvector $\one{\level}$) and 
$v - c$ with multiplicity $\level - 1$ (eigen-space $\one{\level}^\perp$). 
Notice that a CS matrix is positive definite 
for a range of negative values of its correlation term.\\
 
Then we consider the following Gaussian model:
\begin{equation} \label{eq:CSmodel}
\responseAnova_\levelIndex = \groupEffect + \levelEffect_\levelIndex, \qquad \levelIndex = 1, \dots, \level
\end{equation}
where $\groupEffect \sim \mathcal{N}(0, \varCentre)$ with $\varCentre > 0$,
and $\levelEffect_1, \dots, \levelEffect_\level$ are i.i.d. random variables from $\mathcal{N}(0, \varIndiv)$,
with $\varIndiv > 0$,
assumed to be independent of $\groupEffect$.

A direct computation shows that the covariance matrix of $\bm{\responseAnova}$ 
is the CS covariance matrix $\compound{\level} {\varCentre + \varIndiv, \varCentre}$.
Clearly this characterizes the subclass of positive definite CS covariance matrices $\compound{\level}{v,c}$ 
such that $c$ is non-negative.
The full parameterization, including negative values of $c$ in the range ${(-(\level-1)^{-1}v, 0)}$,
 can be obtained by restricting the average of level effects to be zero, 
as detailed in the next proposition.

\newtheorem{prop1}{Proposition}
\begin{prop1} \label{prop:CSparam}
When $\bm{\responseAnova}$ and $\bm{\levelEffect}$ are related as in (\ref{eq:CSmodel}), 
the covariance of $\bm{\responseAnova}$ conditional on zero average errors $\overline{\levelEffect}= 0$
is a CS matrix with variance
$v = v_{\groupEffect} + v_{\levelEffect}[1-1/\level]$ and covariance
$c = v_{\groupEffect} - v_{\levelEffect}/\level$. 
Conversely, given a CS covariance matrix~$\mb{C}$ with variance $v$ and covariance $c$,
there exists a representation (\ref{eq:CSmodel}) such that $\mb{C}$ is the covariance of $\bm{\responseAnova}$ 
conditional on zero average errors $\overline{\levelEffect} =0$ 
where $v_\groupEffect = v/\level + c[1-1/\level]$ and $v_{\levelEffect} = v - c$.
\end{prop1}

\subsection{Parameterization of centered covariance matrices} \label{sec:centeredCovMat}
The usage of Model~(\ref{eq:CSmodel}) to describe CS covariance matrices involves Gaussian vectors that sum to zero.
This is linked to centered covariance matrices, i.e. covariance matrices $\varIndivG{}$ such that $\varIndivGbar{} = 0$, 
as detailed in the next proposition. We further give a parameterization of centered covariance matrices.

\newtheorem{prop2}[prop1]{Proposition}
\begin{prop2} \label{prop:covDegenerate}
Let $\varIndivG{}$ be a covariance matrix of size $\level \geq 2$. 
Then, $\varIndivG{}$ is centered iff there exists a Gaussian vector 
$\mb{z}$ on $\mathbb{R}^\level$ such that $\varIndivG{} = \Cov(\mb{z} \vert {\overline{z}=0})$.
In that case, let $\constrastMat{}$ be a $\level \times (\level-1)$ matrix whose columns form an orthonormal basis of $\one{\level}^\perp$.
Then $\varIndivG{}$ is written in an unique way
\begin{equation} \label{eq:FgParam}
\varIndivG{} = \constrastMat{} \varZeroAverageRed{} \constrastMat{}^\top
\end{equation}
where $\varZeroAverageRed{}$ is a covariance matrix of size $\level-1$.\\
In particular if $\varIndivG{} = v[\I{\level} - \level^{-1}\J{\level}]$ is a centered CS covariance matrix, then $\varZeroAverageRed{} = v \I{\level-1}$, 
and we can choose $\mb{z} \sim \mathcal{N}(0, v \I{\level})$. 
\end{prop2}

The choice of $\constrastMat{}$ in Prop.~\ref{prop:covDegenerate} is free, 
and can be obtained by normalizing the columns of a $\level \times (\level-1)$ Helmert contrast matrix 
(\cite{Venable_Ripley_MASS}, \S 6.2.):
$$\begin{bmatrix}
-1 & -1 & -1 & \cdots & -1 \\ 
 1 & -1 & -1 & \cdots & -1 \\
 0 &  2 & -1 & \cdots & -1 \\
\vdots  &  0 &  3 & \ddots & \vdots \\
\vdots  &  \vdots  &  \ddots  & \ddots  &  -1  \\
 0 &  0 &  \cdots & 0 & \level-1 \\
\end{bmatrix} $$

\subsection{A hierarchical Gaussian model for GCS block covariance matrices} \label{sec:hierarchicalGP}
Let us now return to the general case, where the levels of $u$ are partitioned in $\nbGroup$ groups.
It will be convenient to use the hierarchical notation $g/\levelIndex$, indicating that $\levelIndex$ belongs to the group $\group_g$.
Then, we consider the following hierarchical Gaussian model:
\begin{equation} \label{eq:blockMatModel}
\responseAnova_{g/\levelIndex} = \groupEffect_g + \levelEffect_{g/\levelIndex}, \qquad g = 1, \dots, \nbGroup, \quad \levelIndex \in \group_g
\end{equation}
where for each $g$ the random variable $\groupEffect_g$ represent the \textit{effect of the group $g$},
and the random variables $\levelEffect_{g/1}, \dots, \levelEffect_{g/n_g}$ represent the \textit{effects of the levels} in this group.
We assume that $\bm{\groupEffect}$ is normal $\mathcal{N}(0, \varCentreGroup)$,
and all the $\bm{\levelEffect}_{g/.}$ are normal $\mathcal{N}(0, \varIndivG{g})$.
We also assume that $\bm{\levelEffect}_{1 / .}, \dots, \bm{\levelEffect}_{\nbGroup / .}$ are independent,
and independent of $\bm{\groupEffect}$.\\ 
Notice that, up to centering conditions on $\bm{\levelEffect}_{g/.}$ that will be considered next, 
$\groupEffect_g$ is the mean of group $g$.
Hence, 
$\varCentreGroup$ is interpreted as the \textit{between group means} covariance.
Similarly, $\bm{\levelEffect}_{g/.}$ is the within-group effect around the group mean. 
This justifies the notations $\varCentreGroup$ and $\varIndivG{g}$.\\

As an extension of Prop.~\ref{prop:CSparam}, the next results
show that (\ref{eq:blockMatModel}) 
gives a one-to-one parameterization of valid GCS block covariance matrices, 
under the additional assumption that the average of level effects is zero in each group.

\newtheorem{thm1}{Theorem}
\begin{thm1} \label{prop:blockParam}
The covariance matrix of $\bm{\responseAnova}$ conditional on 
$\{ \overline{\levelEffect_{g / .}} = 0, \, g=1, \dots, \nbGroup\}$
is a GCS block matrix with, for all $g, g' \in \{1, \dots, \nbGroup \}$: 
\begin{equation} \label{eq:blockCov}
\begin{split}
\mb{\within}_g &= \varCentreGroupCoef{g, g} \J{n_g} + \varIndivG{g}, \\
\between_{g, g'} &= \varCentreGroupCoef{g, g'} \J{n_g, n_{g'}}, 
\end{split}
\end{equation}
where $\varIndivG{g}$ is a centered positive semidefinite matrix equal to $\Cov( \bm{\levelEffect}_{g/.} \vert \overline{\levelEffect_{g / .}} = 0)$. 
Conversely, let $\mb{\matCovCat}$ be a positive semidefinite GCS block matrix.
Then there exists a representation~(\ref{eq:blockMatModel}) 
such that $\mb{\matCovCat}$ is the covariance of $\bm{\responseAnova}$ conditional on zero average errors 
$\overline{\levelEffect_{g/.}} = 0, (g = 1, \dots, \nbGroup)$, with:
\begin{eqnarray*} 
\varCentreGroup &=& \widetilde{\mb{\matCovCat}} , \\
\Cov( \bm{\levelEffect}_{g/.} \vert \overline{\levelEffect_{g / .}} = 0) &=& \mb{\within}_g - \overline{\within_g}  \J{n_g},
\end{eqnarray*}
where $\widetilde{\mb{\matCovCat}}$ is the $\nbGroup \times \nbGroup$ matrix obtained by averaging each block of $\mb{\matCovCat}$.
\end{thm1}

\newtheorem{prop4}{Corollary}
\begin{prop4} \label{prop:blockCSparam}
Positive semidefinite GCS block matrices with CS diagonal blocks 
exactly correspond to covariance matrices of $\bm{\responseAnova}$ in (\ref{eq:blockMatModel}) 
conditional on the $\nbGroup$ constraints $\overline{\levelEffect_{g/.}} =0$ 
when $\Cov(\bm{\levelEffect}_{g/.}) \propto \I{n_g}$.\\
\end{prop4}

As a by-product, we obtain a simple condition for checking the positive definiteness of GCS block matrices.
Interestingly, it only involves a small matrix whose size is the number of groups.

\newtheorem{thm2}[thm1]{Theorem}
\begin{thm2} \label{prop:TandTbar}
Let $\mb{\matCovCat}$ be a GCS block matrix. Then
\begin{enumerate}
\item[(i)] $\mb{\matCovCat}$ is positive semidefinite if and only if $\widetilde{\mb{\matCovCat}}$ is positive semidefinite.
\item[(ii)] $\mb{\matCovCat}$ is positive definite if and only if $\widetilde{\mb{\matCovCat}}$ is positive definite and the diagonal blocks $\mb{\within}_g$ are positive definite for all $g=1, \dots, \nbGroup$.
\end{enumerate} 
Furthermore, we have
\begin{equation} \label{eq:TandTbarEq}
\mb{\matCovCat} = \mb{X} \widetilde{\mb{\matCovCat}} \mb{X}^\top + 
\textrm{diag}(\mb{\within}_1 - \overline{\within_1} \J{n_1}, \dots, 
\mb{\within}_\nbGroup - \overline{\within_\nbGroup} \J{n_\nbGroup})
\end{equation}
where $\mb{X}$ is the $n \times \nbGroup$ matrix
$$ \mb{X} := 
\begin{pmatrix} 
  \one{n_1} & 0 & \dots & 0 \\
  0 & \one{n_2} & \ddots & \vdots \\
  \vdots & \ddots & \ddots & 0\\
  0 & \dots & 0 & \one{n_\nbGroup}\\
\end{pmatrix}.$$
\end{thm2}

\newtheorem{rem1}[remIdentify]{Remark}
\begin{rem1}
All the results depend on the conditional distribution
$\bm{\levelEffect}_{g/.} \vert { \overline{\levelEffect_{g / .}} = 0}$. 
Thus there is some flexibility in the choice of $\varIndivG{g}$,
since several matrices $\varIndivG{g}$ can lead to the same conditional covariance matrix
$\Cov(\bm{\levelEffect}_{g / .} \vert { \overline{\levelEffect_{g / .}} = 0})$.
\end{rem1}

\newtheorem{rem2}[remIdentify]{Remark}
\begin{rem2}[Groups of size 1]
Theorem~\ref{prop:blockParam} is still valid for groups of size 1.
Indeed if $n_g=1$, then $(\bm{\levelEffect}_{g / .} \vert \overline{\levelEffect_{g / .}} = 0)$ is degenerate and equal to 0.
Thus $\varIndivG{g} = \mb{\within}_g - \overline{\within_g}\J{1} = 0$ is positive semidefinite.\\
\end{rem2}

We end this section with an ``exclusion property'' for groups with strong negative correlation.
A positive semidefinite GCS covariance matrix can exhibit negative within-group correlations,
but this induces limitations on the between-group correlations.
More precisely, the following result shows that if a group 
has the strongest possible negative within-group correlations,
then it must be independent of the others.

\newtheorem{propExclusion}[prop1]{Proposition} 
\begin{propExclusion}[Exclusion property for groups with minimal correlation]
\label{prop:exclusionProperty}
Let $\mb{\matCovCat}$ a GCS covariance matrix, 
and let $\mb{y}$ be a centered Gaussian vector such that $\Cov(\mb{y})=\mb{\matCovCat}$.
Let $g$ be a group number, and denote by $\mb{y}_g$ (resp. $\mb{y}_{-g}$) 
the subvector extracted from $\mb{y}$ whose coordinates are (resp. are not) in group $\group_g$.
Assume that $\mb{\within}_g$ is such that $\overline{\within_g} = 0$.
Then $\mb{y}_g$ is independent of $\mb{y}_{-g}$.
\end{propExclusion}

The condition $\overline{\within_g} = 0$ is linked to minimal correlations. 
Indeed, since $\mb{\within}_g$ is positive semidefinite, $\overline{\within_g} \geq 0$.
The limit case $\overline{\within_g} = 0$ is obtained when negative 
terms of $\mb{\within}_g$ are large enough to compensate positive ones.
As an example, if $\mb{\within}_g$ is a positive semidefinite CS covariance matrix 
with variance $v_g$ and minimal negative covariance $c_g = - (n_g - 1)^{-1}v_g$,
then $\overline{\within_g} = 0$. 

\subsection{Related works} \label{sec:relatedWorks}
The hierarchical model~(\ref{eq:blockMatModel}) 
shares similarities with two-way Bayesian models 
and linear mixed effect models (see e.g. \citet{Lindley_Smith_1972}), 
with Gaussian priors for the effects $\bm{\groupEffect}$ and $\bm{\levelEffect}_{g/.}$.
The centering constraints $\overline{\levelEffect_{g / .}} = 0$ are also standard identifiability conditions in such models.
Furthermore, the particular case of CS covariance matrices 
corresponds to the exchangeable assumption of the corresponding random variables.
Typically, in the framework of linear modelling, Model~(\ref{eq:blockMatModel}) could be written as
$$ y_{g,\ell} = m + \groupEffect_g + \levelEffect_{g, \levelIndex} + \varepsilon_{g, \levelIndex},$$
with additionals grand mean $m$ and errors $\varepsilon_{g, \levelIndex}$.\\

However, if the framework is similar, the goal is different.
In linear modelling, the aim is to quantify the effects by estimating their posterior distribution
$(\bm{\groupEffect}, \bm{\levelEffect}) \vert \mb{y}$.
On the other hand, we aim at investigating the form of the covariance matrix of the response part
$ \groupEffect_g + \levelEffect_{g/\levelIndex}$, 
or, equivalently, the covariance matrix of the \emph{likelihood} 
$\mb{y} \vert \bm{\groupEffect_g } + \bm{\levelEffect_{g/\levelIndex}}.$

\subsection{Guideline for practical usage} \label{sec:summary}
The results of the previous sections show that valid GCS block covariance matrices
can be parameterized by a family of covariance matrices of smaller sizes.
It contains the case where diagonal blocks are CS covariance matrices. 
The algorithm is summarized below.
\begin{enumerate}
\item Generate a covariance matrix $\varCentreGroup$ of size \nbGroup.
\item For all $g=1, \dots, \nbGroup$,\\
If $n_g = 1$, set $\varIndivG{g} = 0$, else:
\begin{itemize}
\item Generate a covariance matrix $\varZeroAverageRed{g}$ of size $n_g-1$. 
\item Compute a centered matrix $\varIndivG{g} = \constrastMat{g} \varZeroAverageRed{g} \constrastMat{g}^\top$, 
where $\constrastMat{g}$ is a $n_g$ by $n_g-1$ matrix whose columns form an orthonormal basis of $\one{n_g}^\perp$. 
\end{itemize} 
\item For all $1 \leq g <  g' \leq \nbGroup$, compute the within-group blocks $\mb{\within}_g$ and between-group blocks $\between_{g,g'}$
by Eq. (\ref{eq:blockCov}).
\end{enumerate} 

In steps 1 and 2, the generator covariance matrices $\varCentreGroup$ and $\varZeroAverageRed{g}$ can be general, 
and obtained by one of the parameterizations of \S \ref{sec:parametricCovMat}.
A direct application of Theorem~\ref{prop:TandTbar} also shows that $\mb{\matCovCat}$ is invertible 
if and only if $\varCentreGroup$ and the $\varZeroAverageRed{g}$'s are invertible (cf. Appendix for details).\\
Furthermore, some specific form, such as CS matrices, can be chosen.
Depending on the number of groups and their sizes, different levels of parsimony can be obtained.
Table~\ref{tab:paramSetting} summarizes some possibilities.\\

Note that the parameterization is for a general covariance matrix,
but not for additional constraint, as for a correlation matrix.
In these situations, one can take advantage of the economic condition  
of Theorem~\ref{prop:TandTbar}:
positive semidefiniteness on $\mb{\matCovCat}$ is always equivalent to 
positive semidefiniteness of the small matrix $\widetilde{\mb{\matCovCat}}$ of size $\nbGroup$. 
This can be handled more easily by non-linear or semidefinite programming algorithms.\\

\begin{table}[h!] 
\begin{center}
\small{
\begin{tabular}{c c | c c | c} 
\hline 
\multicolumn{2}{c |}{\it{Parametric setting}} & \multicolumn{2}{c |}{\it{Resulting form of $\mb{\matCovCat}$}} & \it{Number of parameters} \\ 
$\varZeroAverageRed{g}$ & $\varCentreGroup$ & $\mb{\within}_g$ & $\between_{g,g'}$  &\\ \hline 
$v_{\levelEffect_g} \I{n_g - 1}$ & $\compound{}{v_\groupEffect, c_\groupEffect}$  & $\compound{}{v_g, c_g}$ & $c_{g, g'} \equiv c_\groupEffect$ & $2 \nbGroup +1$ \\ 
$v_{\levelEffect_g} \I{n_g - 1}$ & General & $\compound{}{v_g, c_g}$ & $c_{g, g'}$ & $\frac{\nbGroup(\nbGroup+3)}{2}$ \\ \hline
General & $\compound{}{v_\groupEffect, c_\groupEffect}$ & General & $c_{g, g'} \equiv c_\groupEffect$ & $ 2 + \sum_{g = 1}^\nbGroup \frac{n_g (n_g + 1)}{2}$ \\
General & General & General & $c_{g, g'}$ & $ \frac{\nbGroup(\nbGroup+1)}{2} + \sum_{g = 1}^\nbGroup \frac{n_g (n_g + 1)}{2}$ \\ \hline
\end{tabular}}
\caption{Parameterization details for some valid GCS block covariance matrices $\mb{\matCovCat}$.}
\label{tab:paramSetting}
\end{center}
\end{table}

%% file: example.tex

\subsection*{Example 1}
Consider the deterministic function 
\begin{equation*}
f(x, u) = \cos \left( 7 \pi \frac{x}{2} + p (u) \pi - \frac{u}{20} \right) 
\end{equation*}
with $x \in [0, 1]$, $u \in \lbrace 1, \dots, 13 \rbrace$ and $p (u) = \left(0.4+\frac{u}{15} \right)  \mathbbm{1}_{u>9}$. 
\begin{figure}[h!]
\begin{center}
\includegraphics[width=0.8\textwidth]{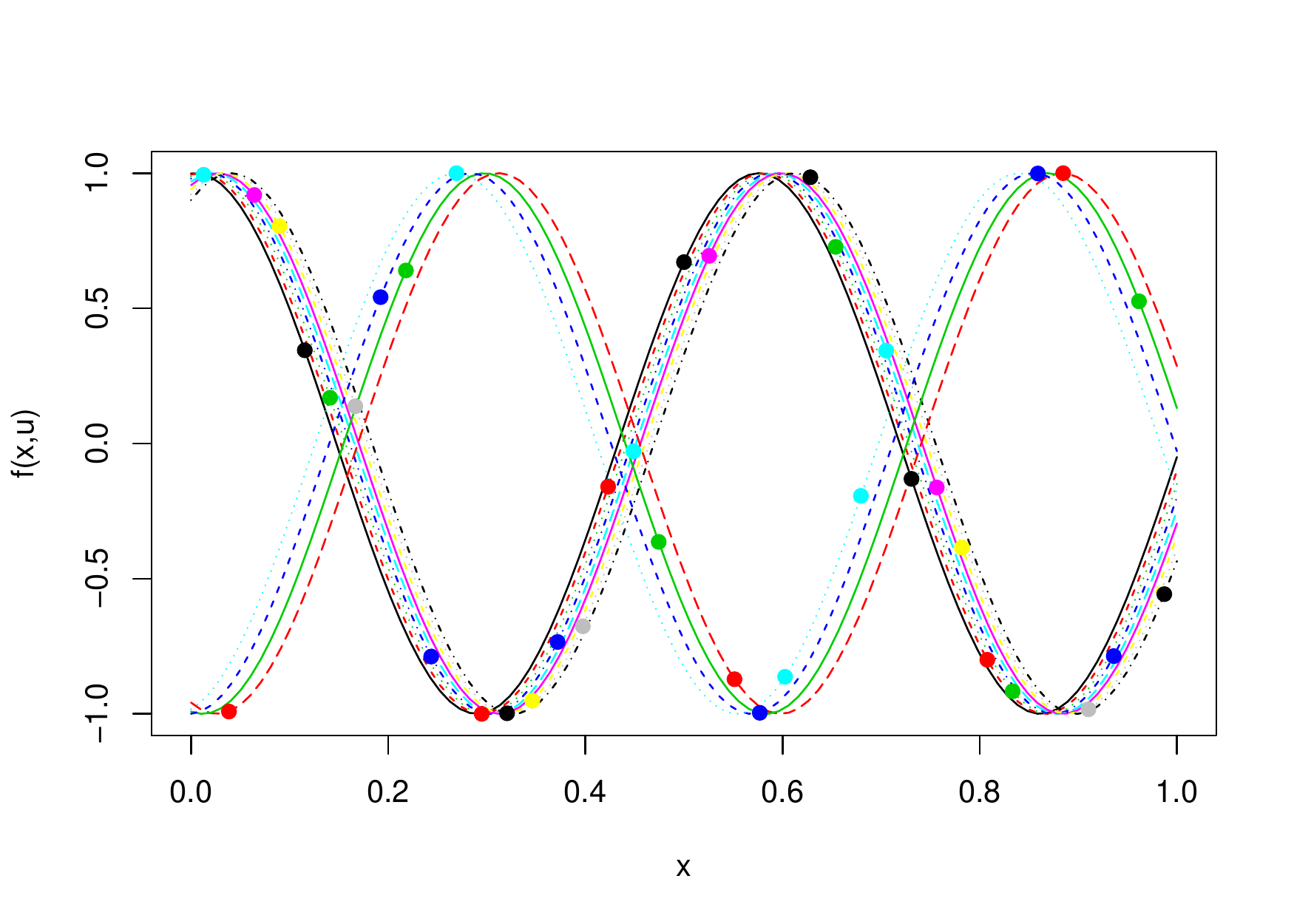}  
\caption{Test function of Example 1. Bullets represent design points.}
\label{toy2}
\end{center}
\end{figure}
As visible in Figure~\ref{toy2}, there are two groups of curves corresponding to levels $\{1, \dots, 9 \}$
and $\{10, \dots, 13 \}$ with strong within-group correlations, and strong negative between-group correlations.\\

\begin{figure}[h] 
\captionsetup[subfigure]{justification=centering}
\centering
\begin{subfigure}{0.32\textwidth}
\centering
\includegraphics[trim = 0cm 2cm 0cm 0cm, width=\textwidth]{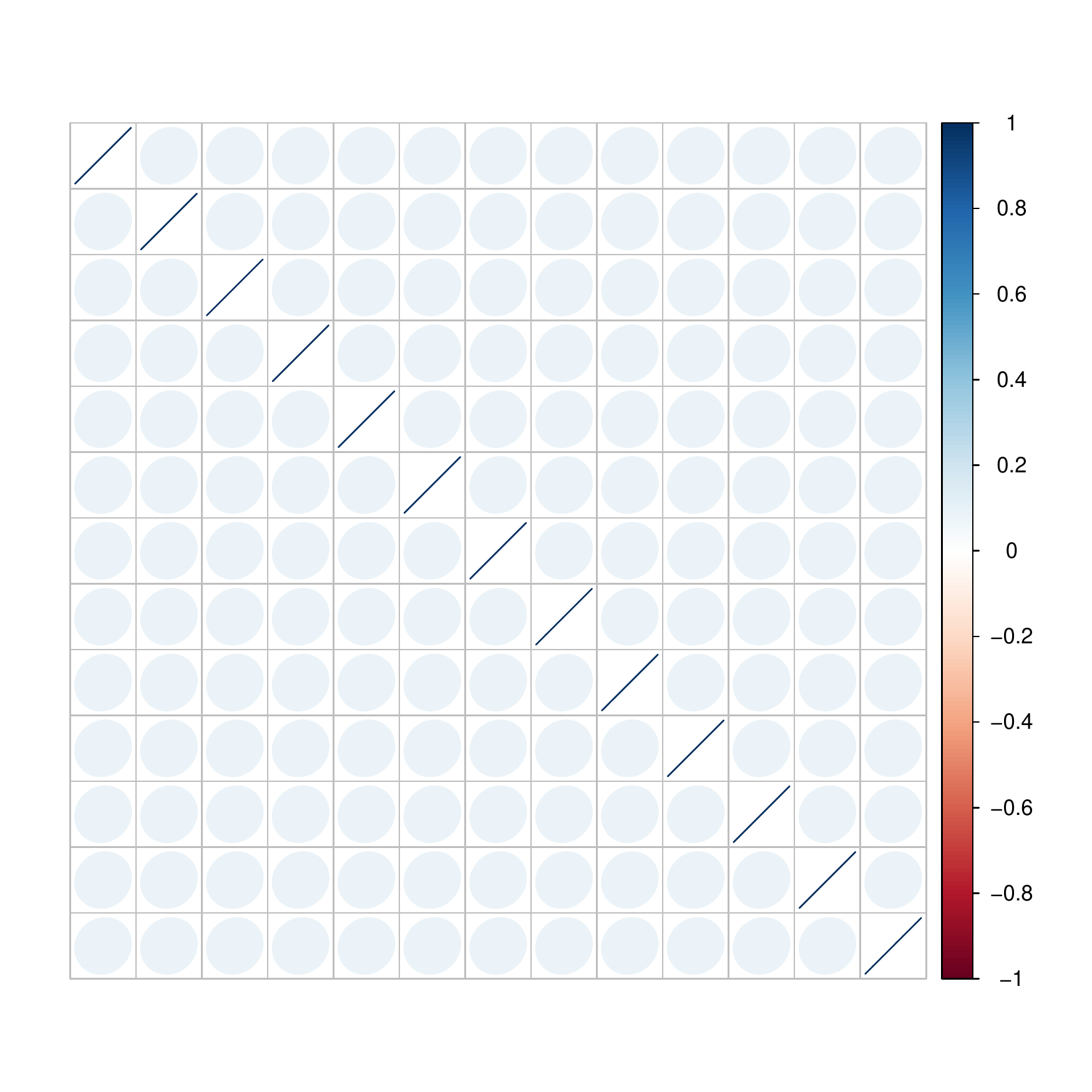} 
\caption*{1 group}
\end{subfigure} \hfill
\begin{subfigure}{0.32\textwidth}
\centering 
\includegraphics[trim = 0cm 2cm 0cm 0cm, width=\textwidth]{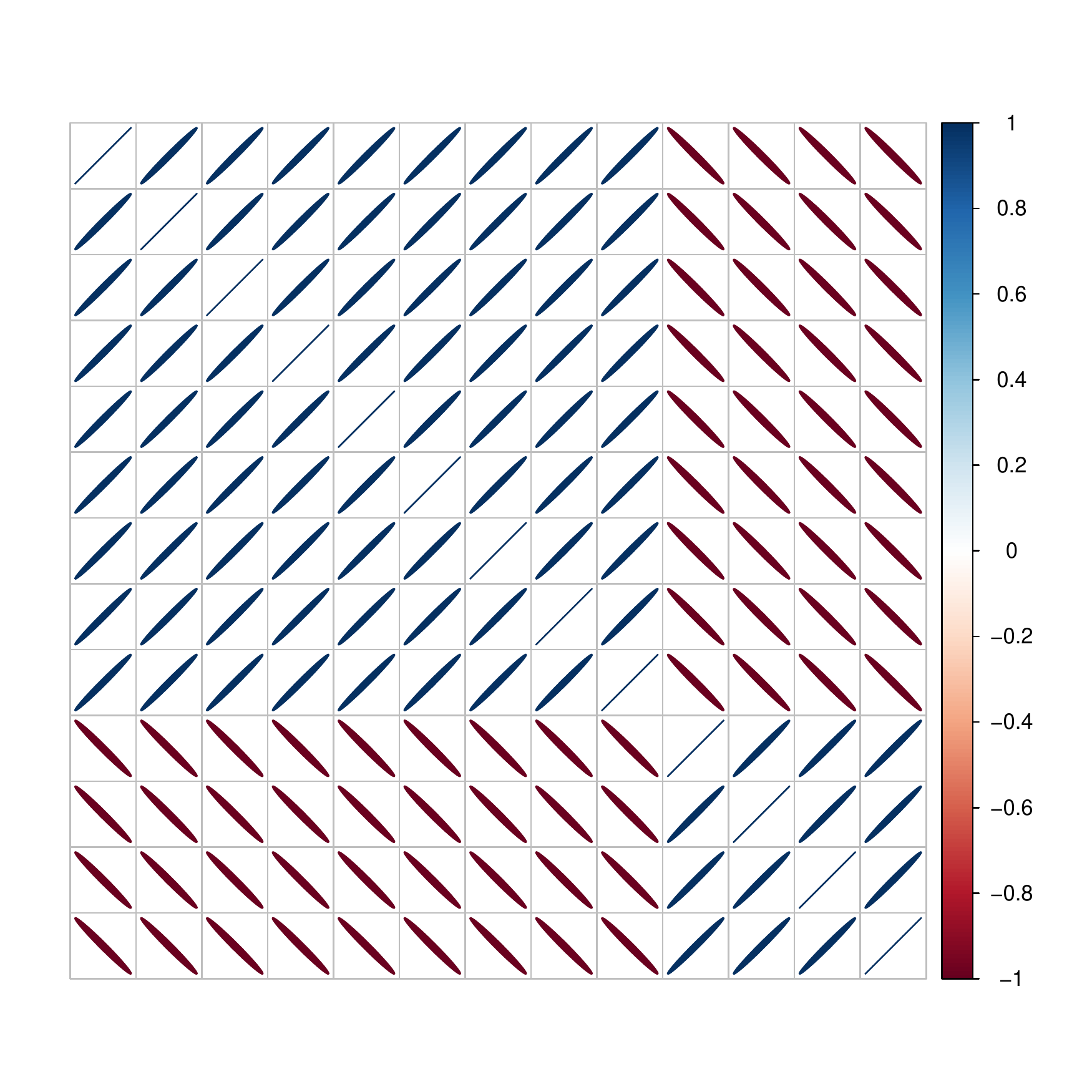}
\caption*{2 groups}
\end{subfigure} \hfill
\begin{subfigure}{0.32\textwidth}
\centering 
\includegraphics[trim = 0cm 2cm 0cm 0cm, width=\textwidth]{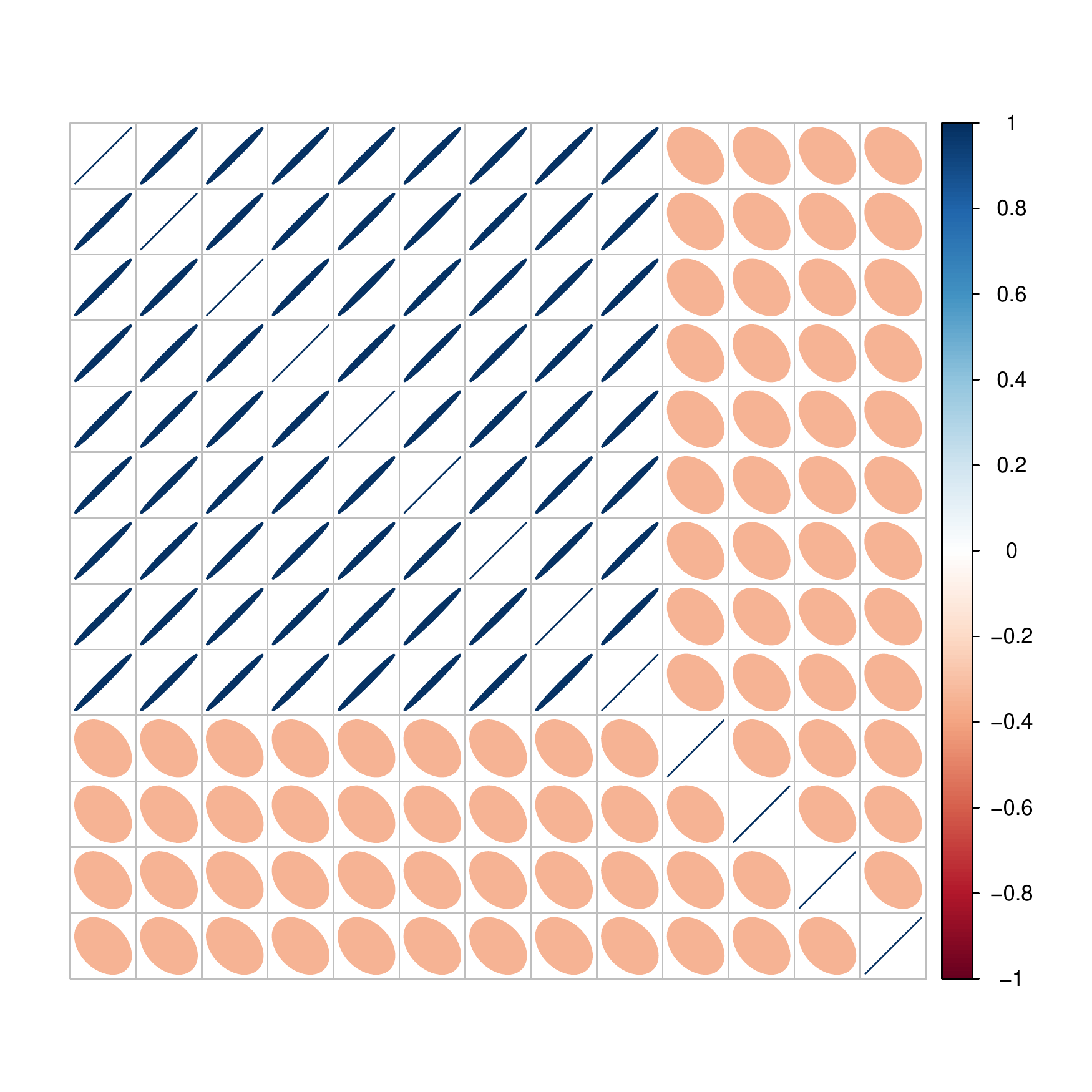} 
\caption*{5 groups (a)}
\end{subfigure}

\begin{subfigure}{0.32\textwidth}
\centering 
\includegraphics[trim = 0cm 2cm 0cm 0cm, width=\textwidth]{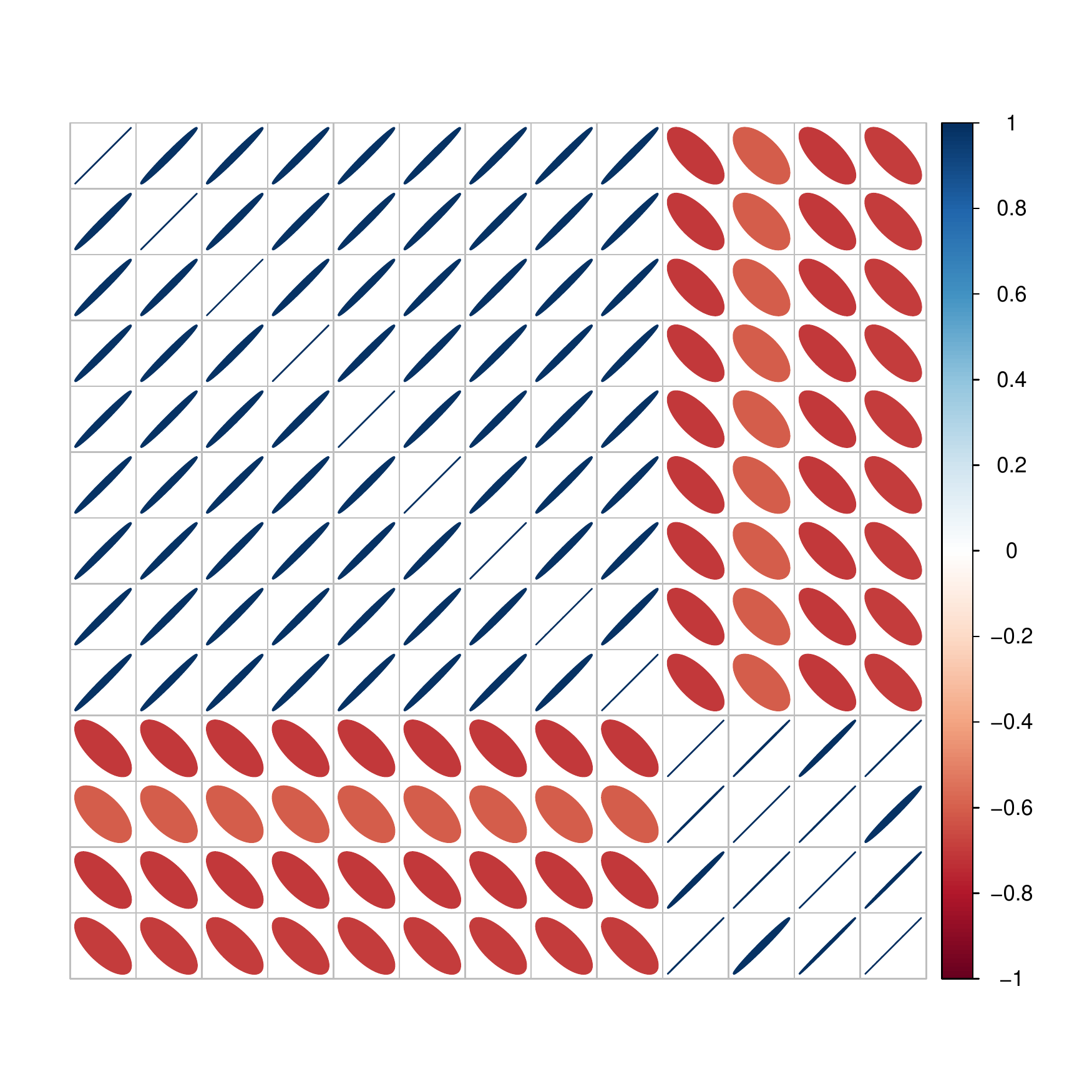} 
\caption*{5 groups (b)}
\end{subfigure} \hfill
\begin{subfigure}{0.32\textwidth}
\centering 
\includegraphics[trim = 0cm 2cm 0cm 0cm, width=\textwidth]{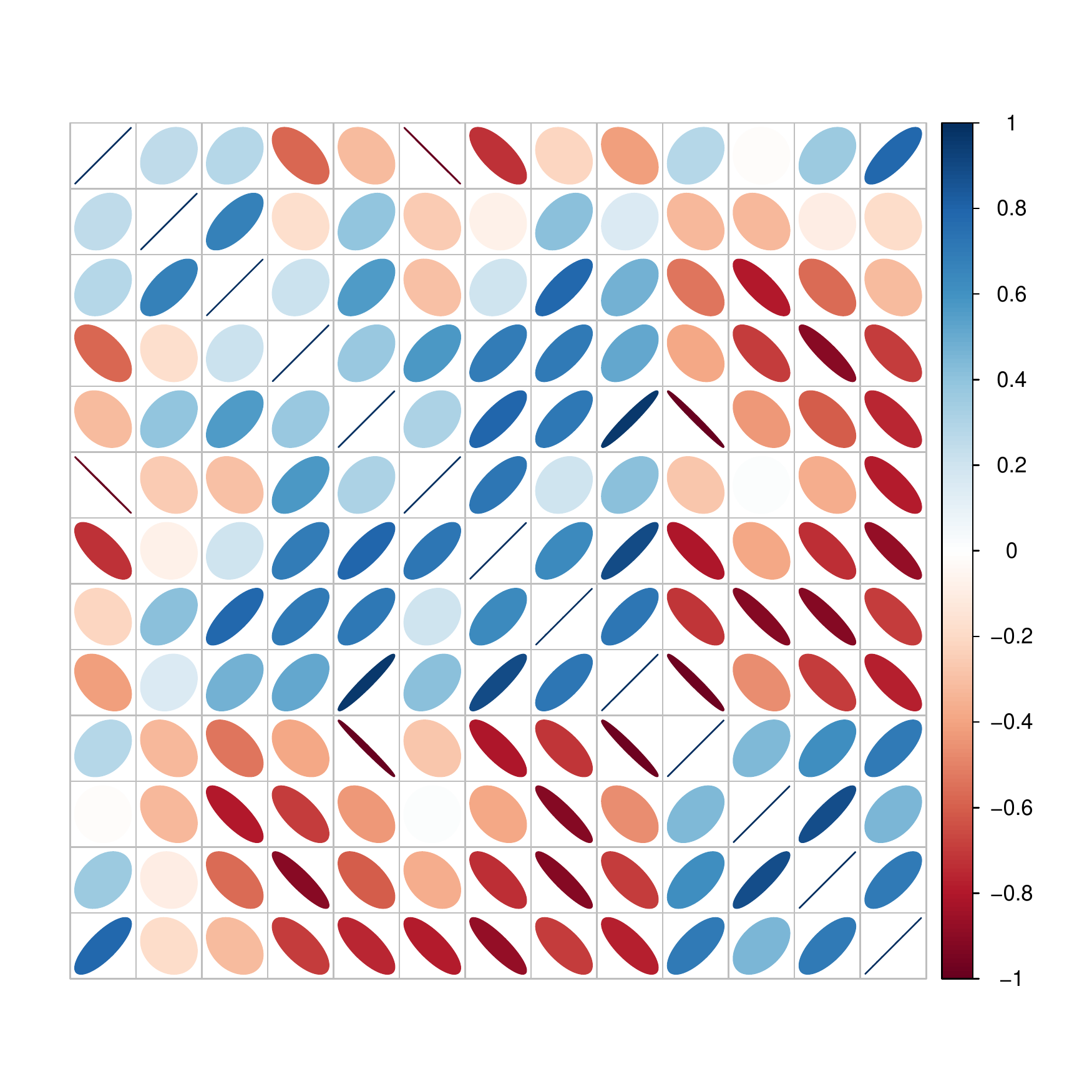}
\caption*{13 groups}
\end{subfigure} \hfill
\begin{subfigure}{0.32\textwidth}
\centering 
\includegraphics[trim = 0cm 2cm 0cm 0cm, width=\textwidth]{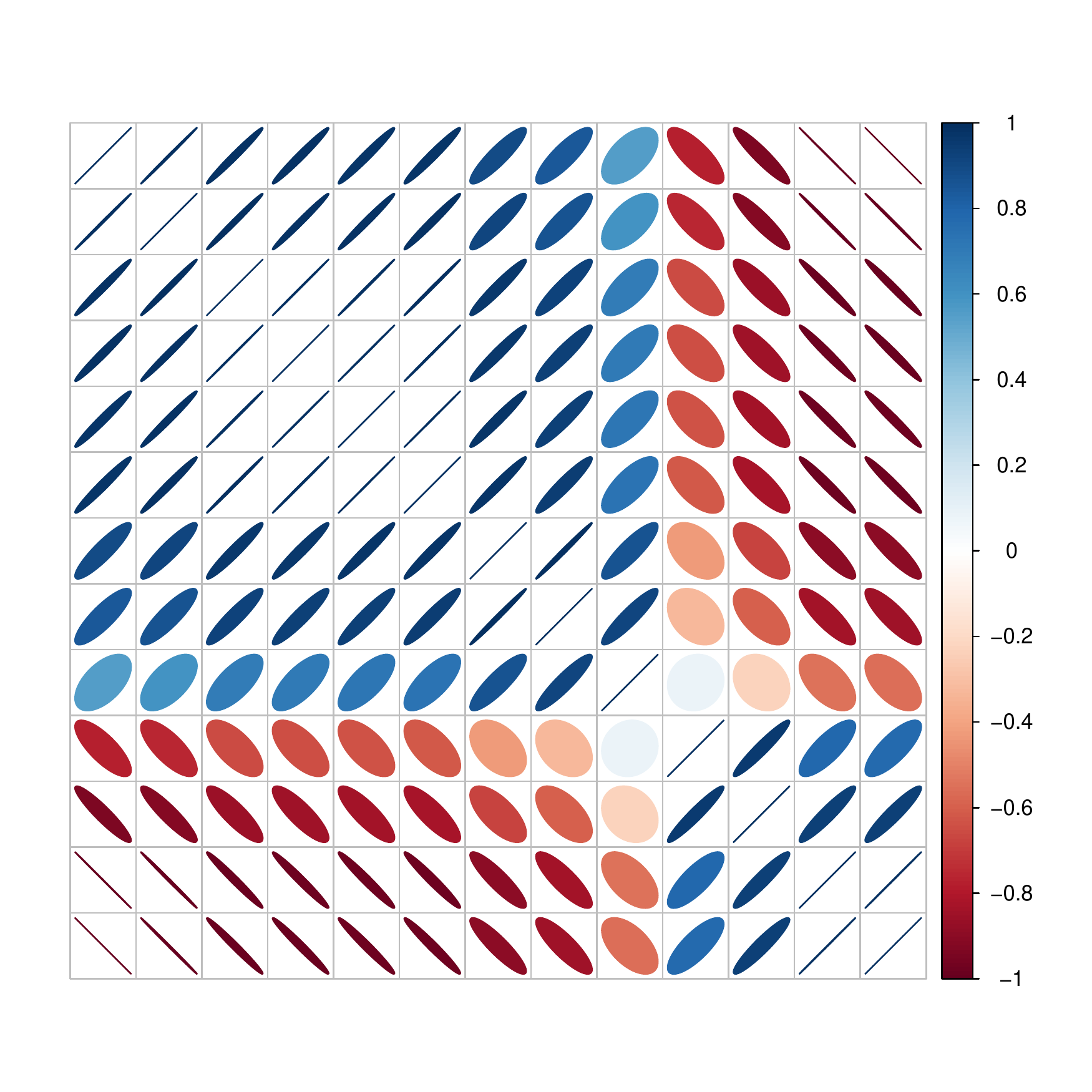} 
\caption*{Ordinal}
\end{subfigure}
\caption{Estimated correlation kernel $k_{\textrm{cat}}$, based on a representative design of experiments (design with median $Q^2$).}
\label{corr_plot_g2}
\end{figure}

We aim at reconstructing $g$ with five GP models based on levels grouping.
The first one uses a CS covariance matrix, corresponding to a single group. 
The second one considers the two groups $\{1, \dots, 9 \}$ and $\{10, \dots , 13\}$. 
The third model, based on the five groups $\{1, \dots, 9 \}$, $\{10\}$, $\{11\}$, $\{12\}$, $\{13\}$, has two variants: (a) when the inter-groups correlation is constant and (b) in the general case.
The fourth model uses the spherical parameterization of $\mb{\matCovCat}$, leading to 13 groups, and the last one considers an ordinal paramaterization for $\mb{\matCovCat}$.
We also compare the result with an ordinal kernel obtained by mapping a piecewise linear cdf $F$ 
into the cosine kernel $k_Z$ of Eq.~\ref{eq:cosKernel} with $\alpha = \pi$.
(see Section~\ref{sec:ordinalKernel}). 
Notice that the choice of $k_Z$ is motivated by recovering negative correlations,
and has no link with the sinusoidal form of the curves of the example.\\

In order to benefit from the strong link between levels, 
we use a design that spreads out the points between levels.
For instance, the information given by $g(0,1)$ may be useful to estimate $g(x,\ell)$ 
at $0$ for a different level $\ell \geq 2$, without computing $g(0,\ell)$. 
More precisely, we have used a (random) sliced Latin hypercube design (SLHD) \citep{qian2012} 
with $3$ points by level, for a total budget of $39$ points.\\
All parameters are estimated by maximum likelihood. 
As the likelihood surface may be multimodal, 
we have launched several optimizations with different starting points chosen at random in the domain.\\
Model accuracy is measured over a test set formed by a regular grid of size $1000$, in terms of $Q^2$ criterion.
The $Q^2$ criterion has a similar expression than $R^2$, but is computed on the test set:
\begin{equation} \label{Q2}
Q^2 = 1 - \frac{\sum_{i}(y_i - \hat{y}_i)^2}{\sum_{i}(y_i - \bar{y})^2},
\end{equation}
where the $y_i$ denote the observations (on the test set), $\bar{y}$ their mean, $\hat{y}_i$ the predictions. 
It is negative if the model performs worst than the mean, 
positive otherwise, and tends to 1 when predictions are close to true values.
Finally, the process is repeated $100$ times, in order to assess the sensitivity of the result to the design.

\paragraph{Results.}
The estimated correlation parameters are shown in Figure~\ref{corr_plot_g2}.
Information about the plots used is given in a last section.
The correlation structure that can be intuited from Figure~\ref{toy2}
is well recovered with two groups and five groups, 
with different between-groups correlations.
The model with thirteen groups involves the estimation of $90$ parameters, which is hard to achieve, especially with $39$ points. 
This is visible in the erratic values of the estimated correlations values, which seem not meaningful.
On the opposite, considering only one group or five groups with a common between-group correlation 
oversimplifies the correlation structure.
Finally, the model using an ordinal kernel recovers the two groups of curves, 
as well as the strong negative correlation between them, 
which is made possible by the choice of the 1-dimensional kernel used in the warping.\\

In Figure~\ref{rmse2}, we can see that the best tradeoff between prediction accuracy 
and parsimony is obtained with two groups. 
whereas it reduces the number of observations by group. 
Notice the rather good performance of the ordinal model, at the cost of a larger number of parameters.

\begin{figure}[h] 
\begin{center}
\includegraphics[width=0.85\textwidth]{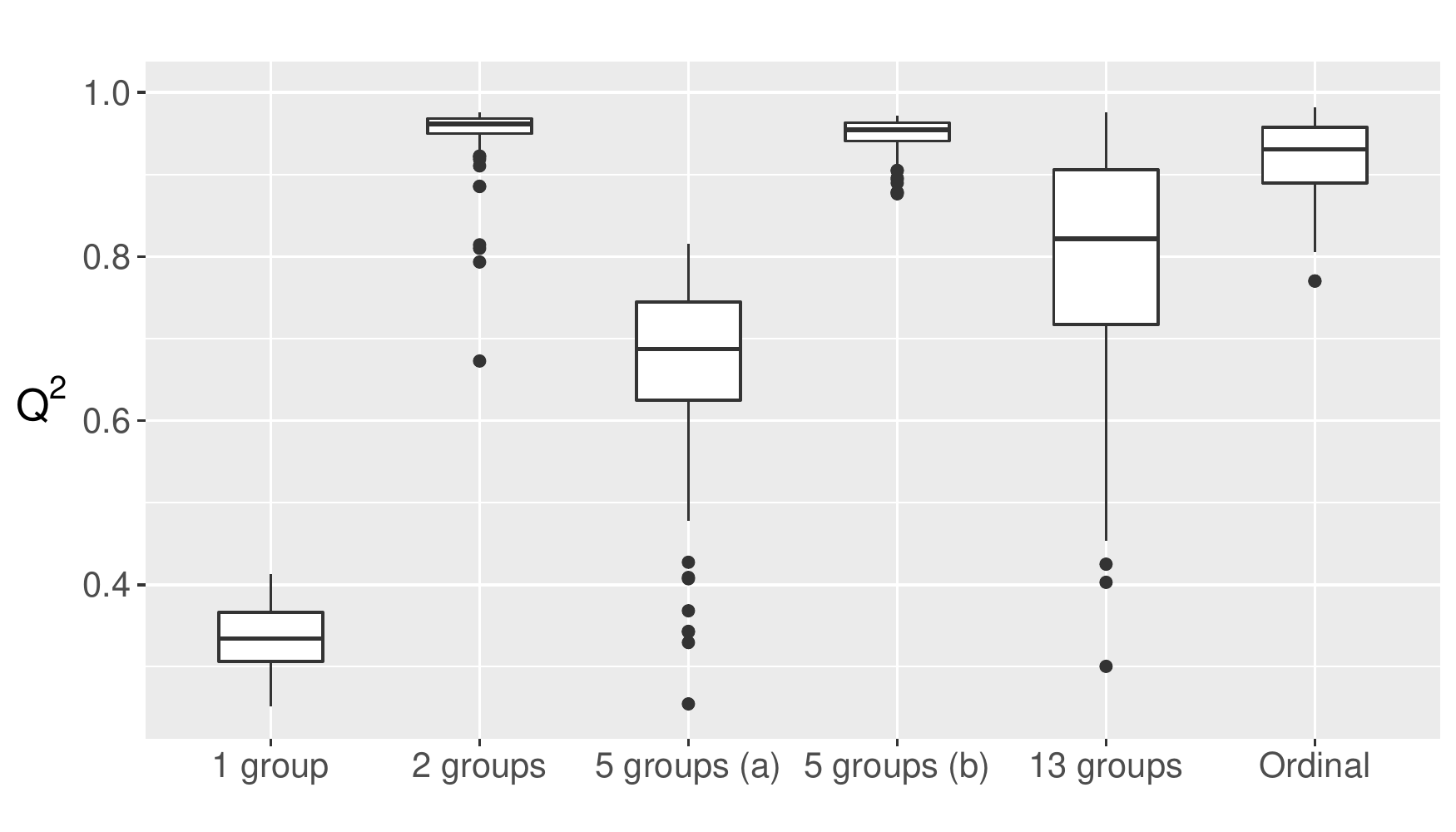} 
\caption{$Q^2$ of six GP models, based on $100$ repetitions of the design.
From left to right: 1 group (CS structure), 2 groups, 5 groups (a : common between-group covariance), 
5 groups (b : general), 13 groups, ordinal.
Number of parameters used (bloxplot order): 5, 7, 10, 19, 90, 16.}
\label{rmse2}
\end{center}
\end{figure} 

\subsection*{Example 2}
We now provide a second example, in order to illustrate the ability of the hierarchical model~(\ref{eq:blockMatModel})
to deal with negative within-group correlations. 
We consider the deterministic function given by:
$$ f(x, u) = \begin{cases} 
   (x + 0.01 (x - 1/2) ^ 2) \times u / 10 & \text{if } u = 1, 2, 3, 4 \\
   0.9 \cos(2 \pi (x + (u - 4) / 20 )) \times \exp(-x)  & \text{if } u = 5, 6, 7 \\
   -0.7 \cos(2 \pi (x + (u - 7) / 20 )) \times \exp(-x)  & \text{if }    u = 8, 9, 10    
  \end{cases}
$$
with $x \in [0, 1]$, $u \in \lbrace 1, \dots, 10 \rbrace$. 
As visible in Figure~\ref{fig:toy_negCor}, the levels can be split in two groups: 
a group of almost linear functions (levels $1-4$), and a group of damped sinusoidal functions (levels $5-10$).
Within the latter group, there are strong negative correlations between levels $5 - 7$ and $8 - 10$.
Hence, the levels could also be split into three groups.\\

\begin{figure}[h!]
\begin{center}
\includegraphics[width=0.8\textwidth]{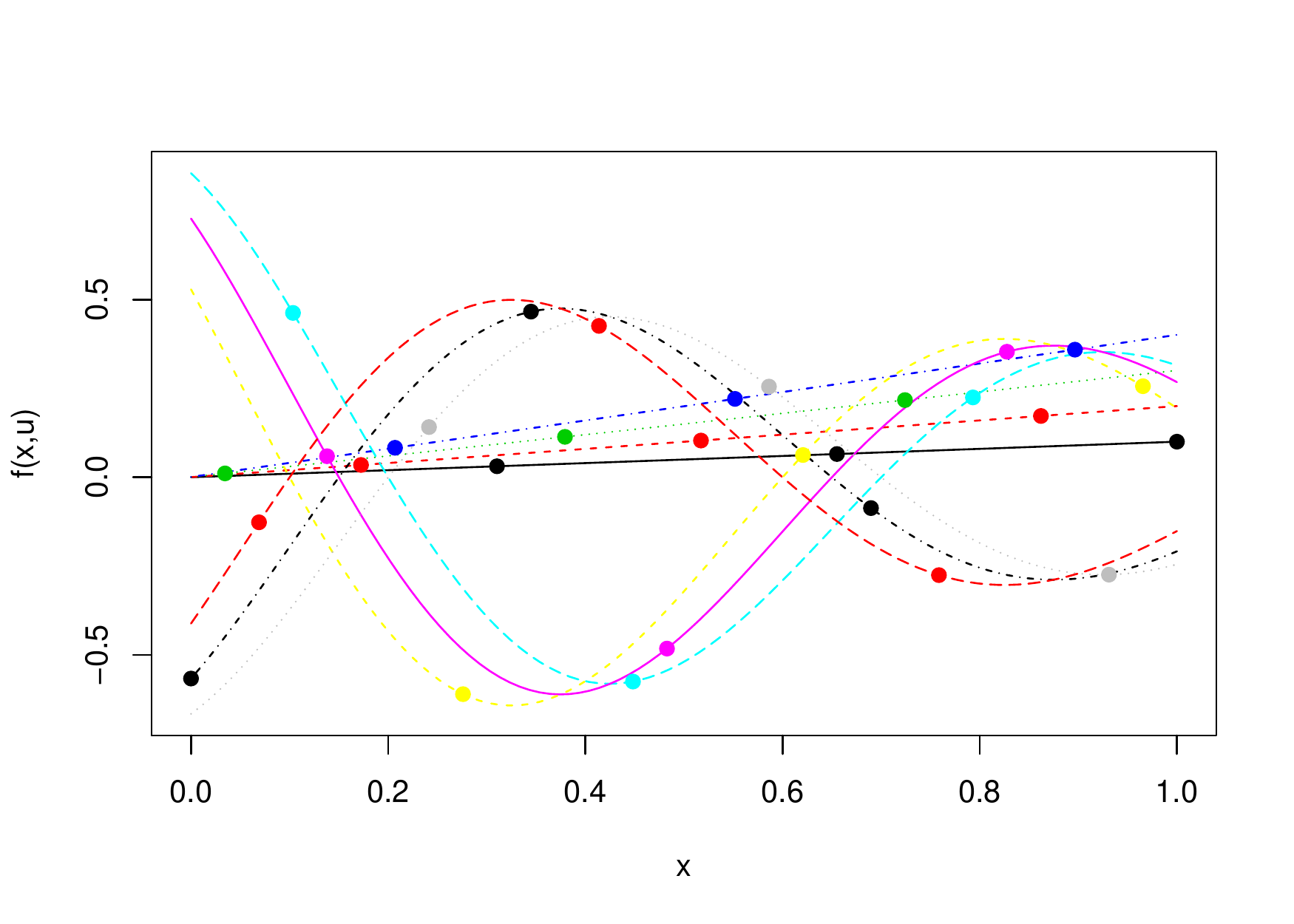}  
\caption{Test function of Example 2. Bullets represent design points.}
\label{fig:toy_negCor}
\end{center}
\end{figure}

In this section, we briefly compare the corresponding GP models:
\begin{itemize}
\item The first model considers the two groups $\{1, \dots, 4 \}$ and $\{5, \dots , 10\}$. 
The within-group structure is CS for the first group (linear functions). 
But a general structure is chosen for the second one (sinusoids), 
in order to capture its complex covariance structure.
\item The second model is based on the three groups $\{1, \dots, 4\}$, $\{5, \dots , 7\}$ and $\{8, \dots , 10\}$.
The within-group structure is CS, and the between-group covariance is general.
\end{itemize}
For simplicity, we consider a single stratified design of experiments 
extracted from a sequence of regularly spaced points, with $m = 3$ points per level.
The other settings are the same as in Example 1.\\

The estimated correlation parameters are shown in Figure~\ref{fig:corplot_exNegCor}.
The correlation structure that can be intuited from Figure~\ref{fig:toy_negCor}
is well recovered by the two models.
However, in the case of two groups, the estimated between-group correlation is nearly zero.
This is an illustration of the exclusion property (Prop.~\ref{prop:exclusionProperty}).
Indeed, due to the strong negative (estimated) correlations within the second group, 
we have $\overline{\within_2} \approx 0$, 
which induces a small correlation between the other group.
In this example, the model with three groups may be more appropriate, 
which seems confirmed by the larger $Q^2$ value of $0.94$ 
(compared to $0.88$ for two groups).
Nevertheless, it is nice to see that, starting from a smaller number of groups,
the correlation plot detects the two subgroups of sinusoids.

\begin{figure}[h] 
\captionsetup[subfigure]{justification=centering}
\centering
\begin{subfigure}{0.35\textwidth}
\centering
\includegraphics[trim = 0cm 1cm 0cm 0cm, width=\textwidth]{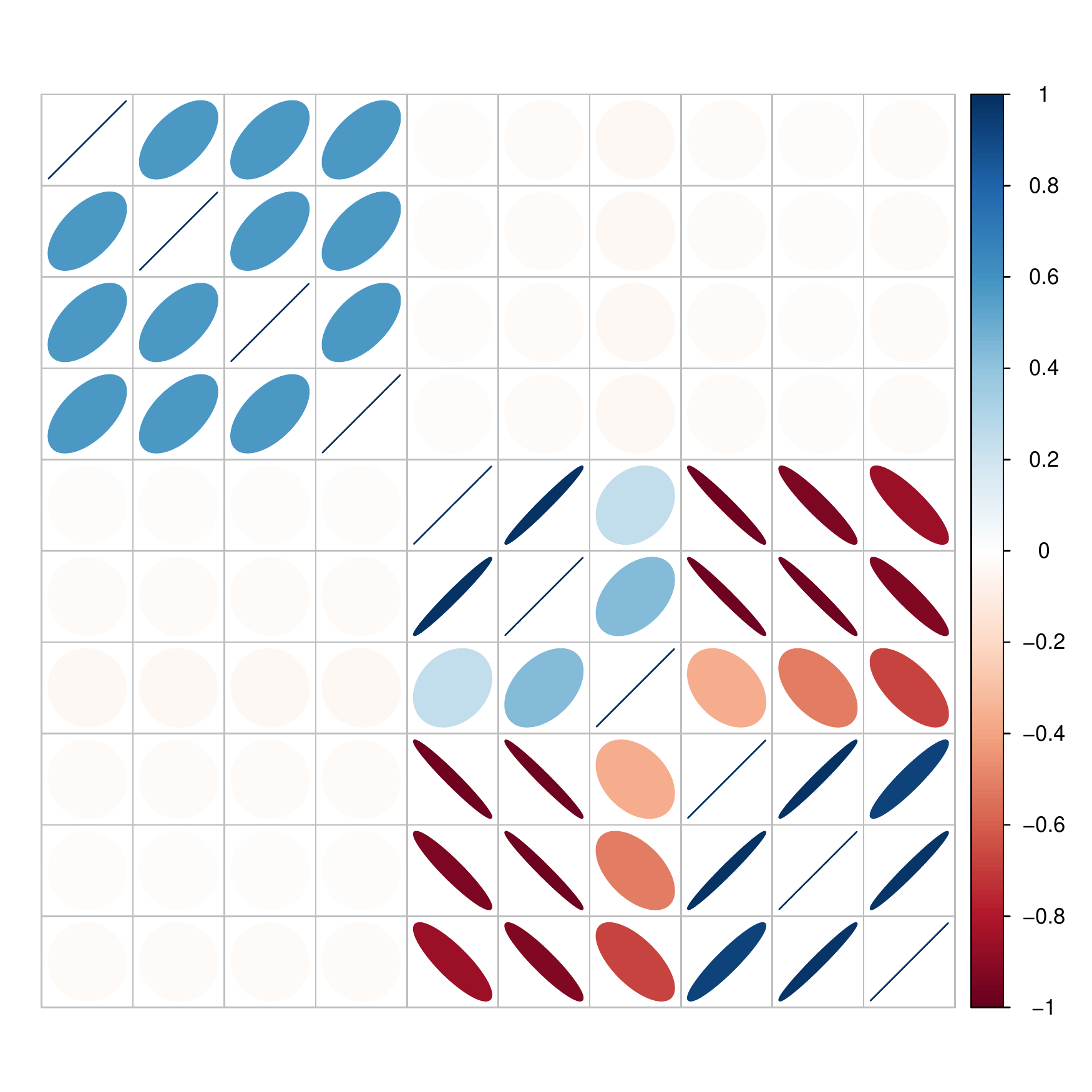} 
\caption*{2 groups}
\end{subfigure} 
\hspace{1cm}
\begin{subfigure}{0.35\textwidth}
\centering 
\includegraphics[trim = 0cm 1cm 0cm 0cm, width=\textwidth]{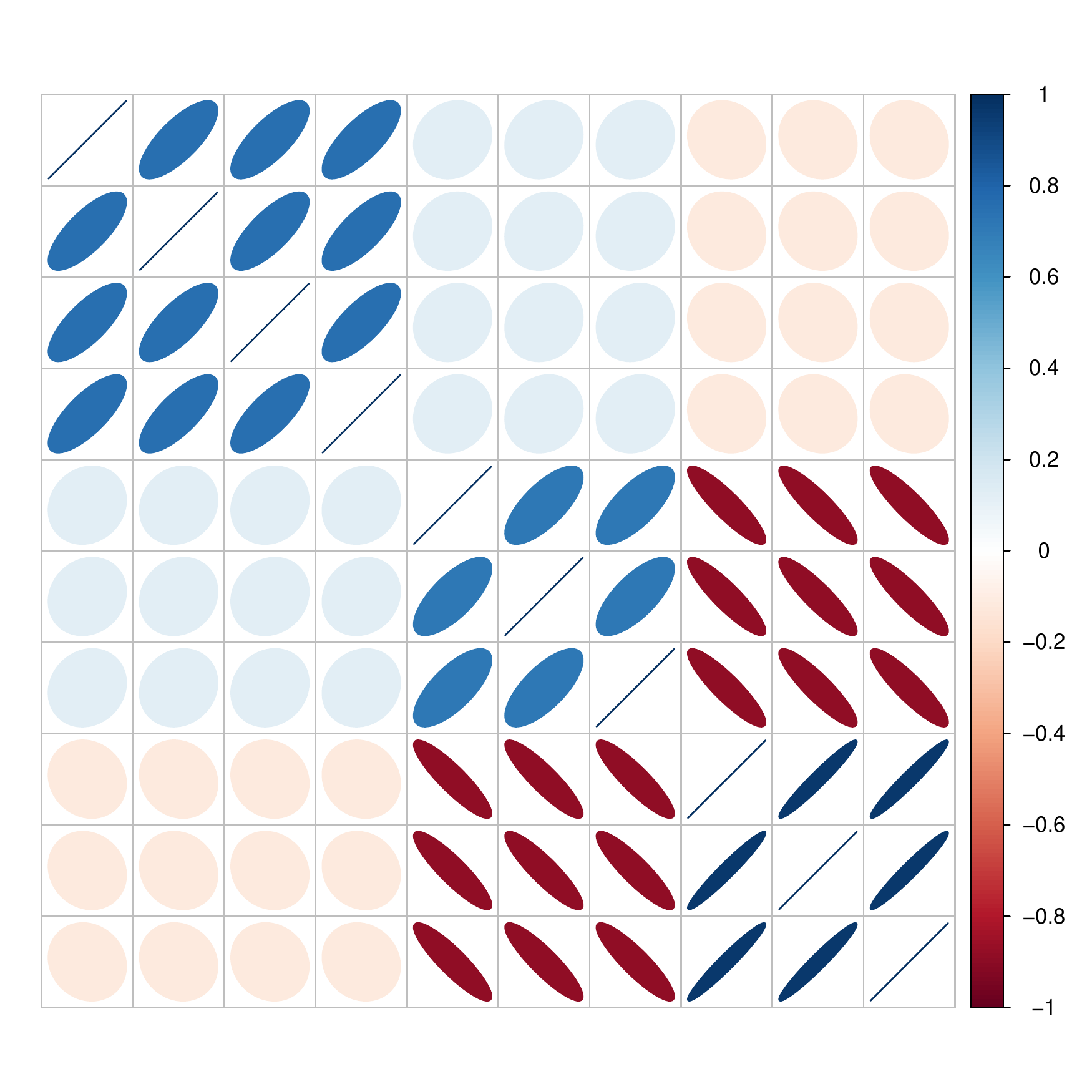}
\caption*{3 groups}
\end{subfigure} 
\caption{Estimated correlation kernel $k_{\textrm{cat}}$ for the two GP models.}
\label{fig:corplot_exNegCor}
\end{figure}

\subsection*{Discussion}
It is remarkable that complex correlation structures can be recovered with few data points per level.
Indeed, even if the correlation structure can be intuited from the whole curves $f(x, u)$,
this complete information is of course not available, 
and the estimation is done with only three points per level in the previous examples. 
The reasons may be twofold.
Firstly, the model is parametric, which may provide more accurate estimation for few data, 
provided that the model is relevant (in particular if the correct groups are given).
Secondly, the global amount of information available may be large enough,
since the small number of points per levels is compensated by the quite large number of levels.
Consequently, when there are few levels, one may need to increase the number of points per level 
in order to reach an acceptable level of accuracy.\\

The sinusoidal form of the toy functions $x \mapsto f(x, u)$ has been chosen 
to illustrate the power of GP models in modelling complex functions. 
Other complex forms could have been used.
An alternative would have been to use sample paths of GP models with group kernels rather than deterministic functions.
Finally, notice that correlation plots may be used with care in general, 
since they do not account for estimation error on correlations.

%% file: application.tex

\subsection{Position of the problem}
\input{presentation_appli}

\subsection{Model settings}
For pedagogical purpose, a dataset of large size $N = 5076$ has been computed with the MNCP code.
The construction of the design of experiments was guided by the categorical inputs, 
such that each of the $6 \times 3 \times 94 = N/3$ combinations of levels appears $3$ times.
It was completed by a Latin hypercube of size $N$ to define the values of the four continuous inputs.\\
From this full dataset, a training set of size $n = 3 \times 94 = 282$ is extracted by selecting at random $3$ observations by chemical element.
The remaining $N - n$ points serve as test set.  

\begin{figure}[h] 
\centering
\begin{subfigure}{0.45\textwidth}
\includegraphics[width=\textwidth]{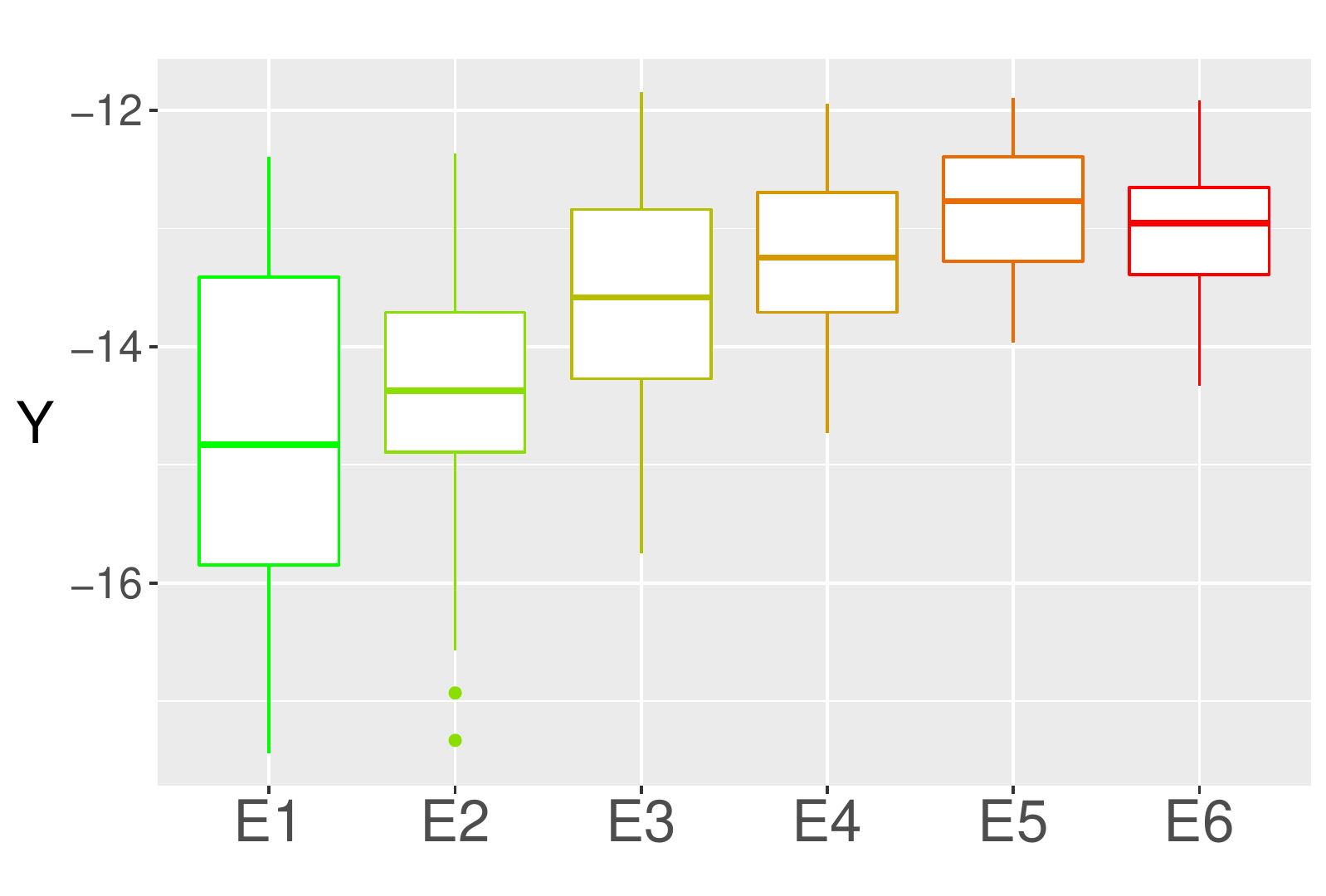}
\end{subfigure}%
\hspace*{0.05\linewidth}
\begin{subfigure}[h]{0.45\textwidth}
\includegraphics[width=\textwidth]{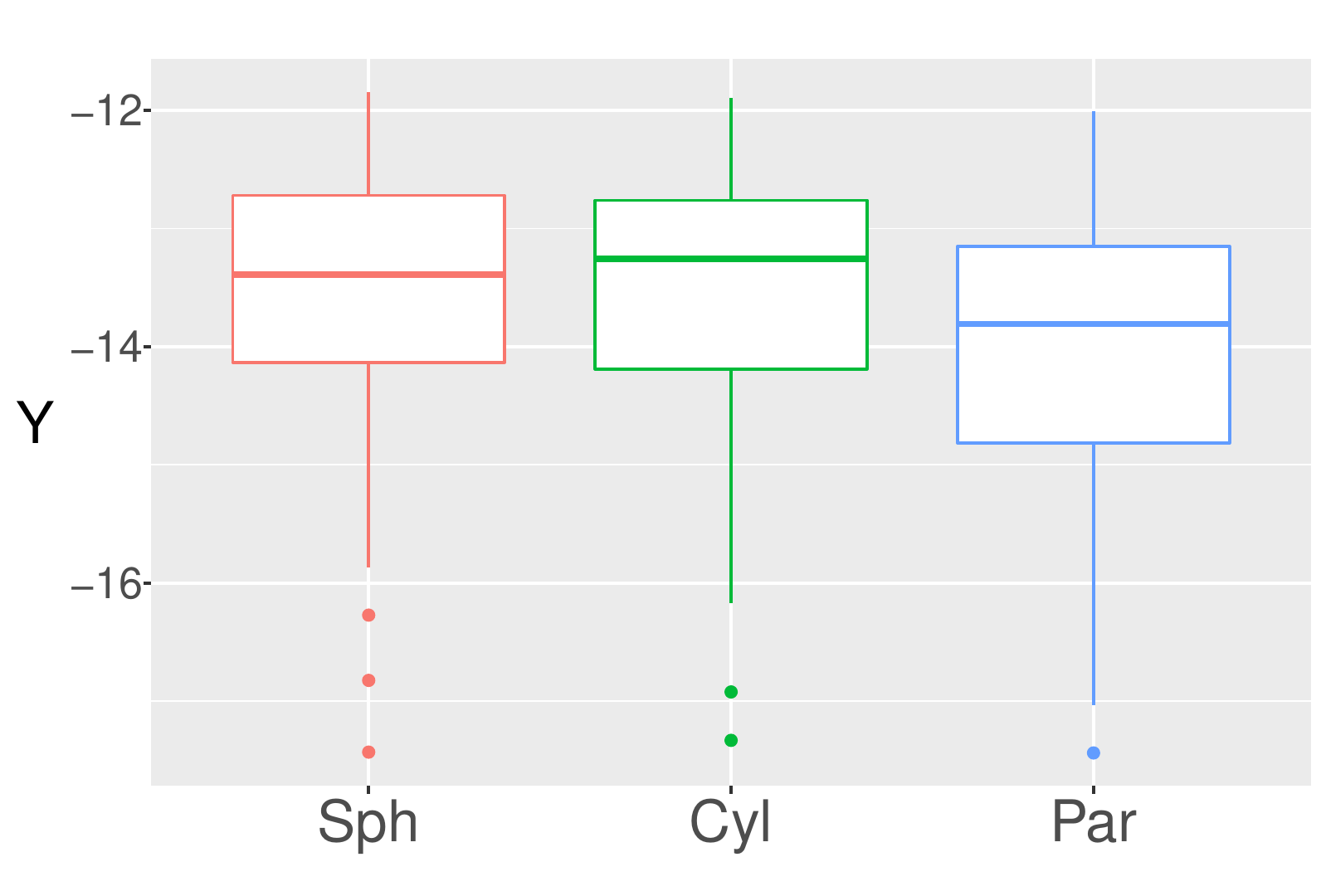}
\end{subfigure}%
\caption{$Y$ in function of the energy (left) and geometric shape (right).}
\label{fig:doe1}
\end{figure}

Model settings are now motivated by a graphical analysis.
In Figure~\ref{fig:doe1}, the output is displayed in function of the energy and the geometric shape. 
We observe that successive energy levels correspond to close values. 
This fact confirms that the energy is ordinal and we use the warped kernel defined by Eq. (\ref{eq:ordinalKernel}). 
The influence of the geometric shape is less obvious, and we have chosen an exchangeable (CS) covariance structure for it. 

\noindent
In Figure~\ref{fig:doe2}, $Y$ is displayed in function of the $94$ chemical elements, ordered by atomic number. 
Two important facts are the high number of levels and heteroscedasticity. 
For this purpose, the $94$ chemical elements are divided into $5$ groups, provided by expert knowledge and represented by colors.
This partition suggests to use a group kernel of the form (\ref{eq:blockCorMat}), 
where the within-group blocks $\within_g$ are CS covariance matrices.
In order to handle heteroscedasticity, the variance of $\within_g$ is assumed to depend on the group number $g$.
\begin{figure}[h] 
\begin{center}
\includegraphics[width=0.5\textwidth]{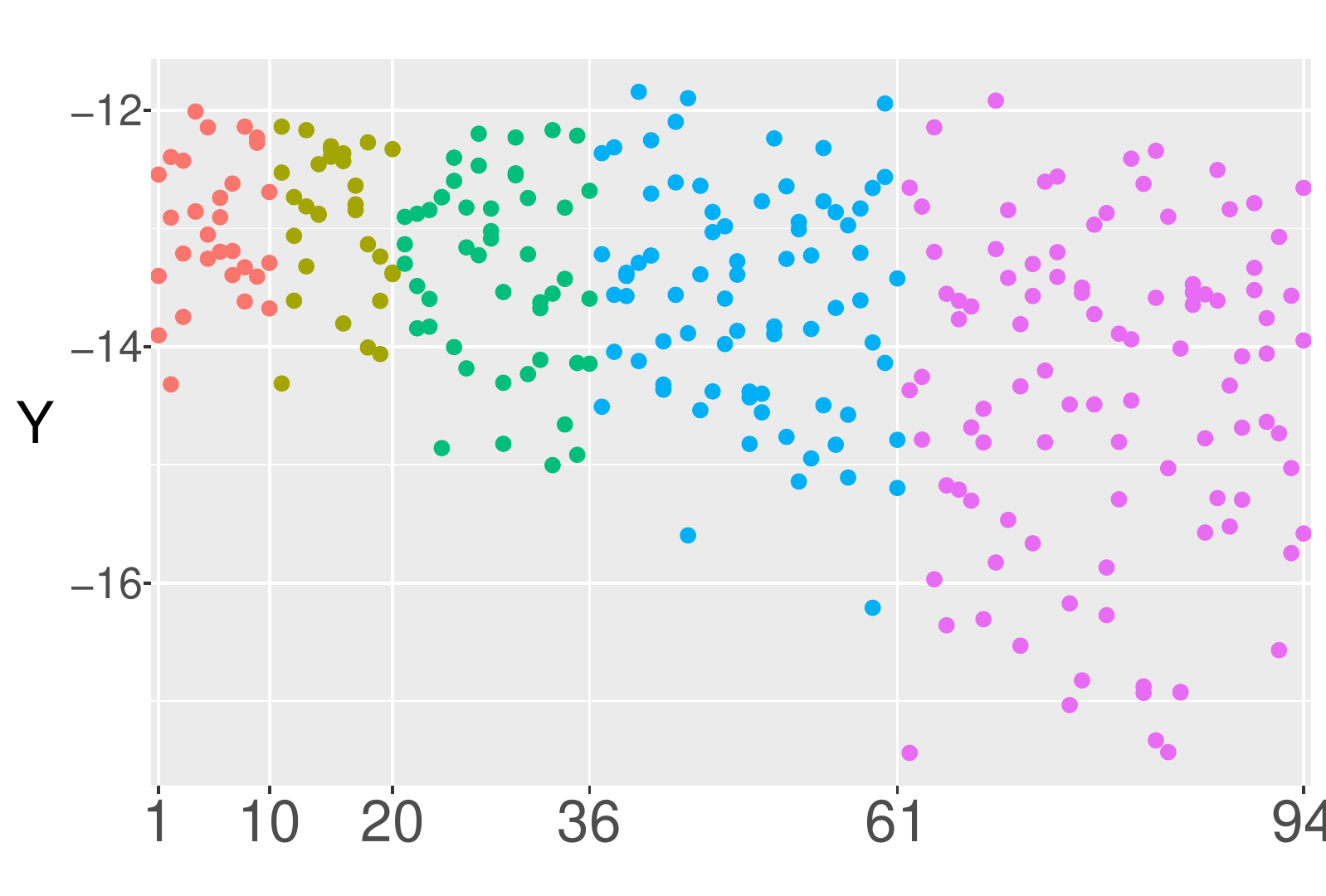}
\caption{Y in function of chemical elements, ordered by atomic number.}
\label{fig:doe2}
\end{center}
\end{figure} 

The influence of continuous variables can be observed by panels (not represented), and does not reveal useful information for our purpose.  
A Mat\'ern $5/2$ kernel is set for all continuous inputs, 
as we expect the output to be a regular function of the continuous inputs. 
Indeed, for this kernel, the corresponding Gaussian process is two times mean-square differentiable.\\
Finally, three candidate kernels for $\mb{w}$ are obtained by combining the kernels of input variables defined above, 
by sum, product or ANOVA (see Section~\ref{sec:background}).

\subsection{Results}
Following the model settings detailed above, 
Figure~\ref{rmse_benchmark}, Panel 3, presents the results obtained with $60$ random designs of size $n$ and three operations on kernels.
Furthermore, we have implemented three other kernels for the chemical element, in order to compare other model choices for this categorical input.
In the first panel, we grouped all the 94 levels in a single group.
In the second one, we kept the 5-group kernel but forced the between-group covariances to have a common value.
Finally, in the fourth panel, we considered that the levels were ordered by their atomic number, 
and used the warped kernel of Eq. (\ref{eq:ordinalKernel}) with a Normal transform.\\
\begin{figure}[h!] 
\centering
\includegraphics[width=\textwidth]{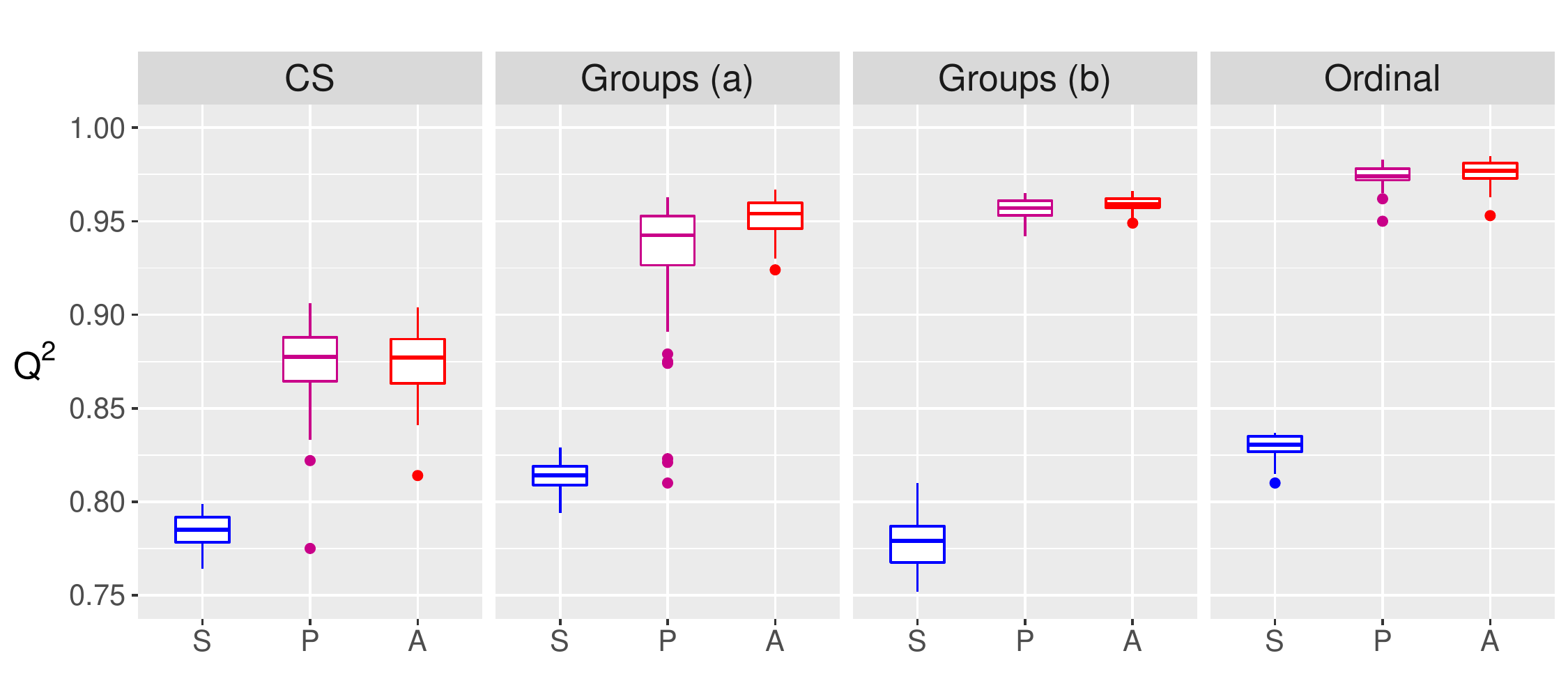} 
\caption{$Q^2$ of several GP models, based on $60$ random designs, corresponding to different model choices for the chemical element.
First panel: Single group. Second panel: 5 groups, with a common between-group covariance. Third panel: 5 groups. Fourth panel: Ordered levels. 
For each panel, three combinations of kernel are tested: sum (S), product (P) or ANOVA (A).
Total number of parameters used (panel order): `prod' = (12, 21, 30, 14), `sum' = `prod' + 6, `anova' = `prod' + 7.}
\label{rmse_benchmark}
\end{figure}  

First, comparing the three operations on kernels, 
we remark that in all the panels, additive kernels provide the worst results. 
This suggests the existence of interactions between different inputs of the simulator. 
Second, the \textrm{ANOVA} combination produces slight improvements, compared to the standard tensor-product,
both in terms of accuracy and stability with respect to design choice.\\
Now, comparing the four panels, we see that gathering the levels in a single group is the least efficient strategy.
The 5-group kernel gives very good performances, especially when the between-group covariances vary freely:
constraining them to be equal degrades the result.
Surprisingly here, the ordinal kernel gives the best performance. 
Indeed, for this application it was not intuitive to the experts that the chemical element can be viewed as an ordinal variable, 
simply sorted by its atomic number. 
This is confirmed by the correlation plots of Figure~\ref{fig:res1b}, 
corresponding to a model with a median $Q^2$ score. 
We can see that the estimated correlations between levels seems to decrease as the difference between levels increases,
an indication that the levels may be ordered by their atomic number.\\
Finally, we report several post-processing results. 
First, the estimated transformation of energy levels (Figure~\ref{fig:res2}, left) is concave and flat near high values,
which corresponds to the behaviour observed in Figure~\ref{fig:doe1} (left panel). 
In addition, the last three levels lead to similar results (Figure~\ref{fig:res2}, right). 
This corresponds to the fact that when the energy is high, the gamma ray almost always crosses the nuclear waste, 
leading to a high value for the output. 
Second, the estimated correlation among the sphere, the cylinder and the parallelepiped is very high ($c = 0.9$, 
Figure~\ref{fig:resGS}). This justifies considering a covariance structure for that categorical input, 
rather than using three independent GP models for all the three levels.

\begin{figure}
\centering 
\begin{subfigure}[h]{0.4\textwidth}
\centering 
\includegraphics[trim = 0cm 2cm 0cm 0cm, width=\textwidth]{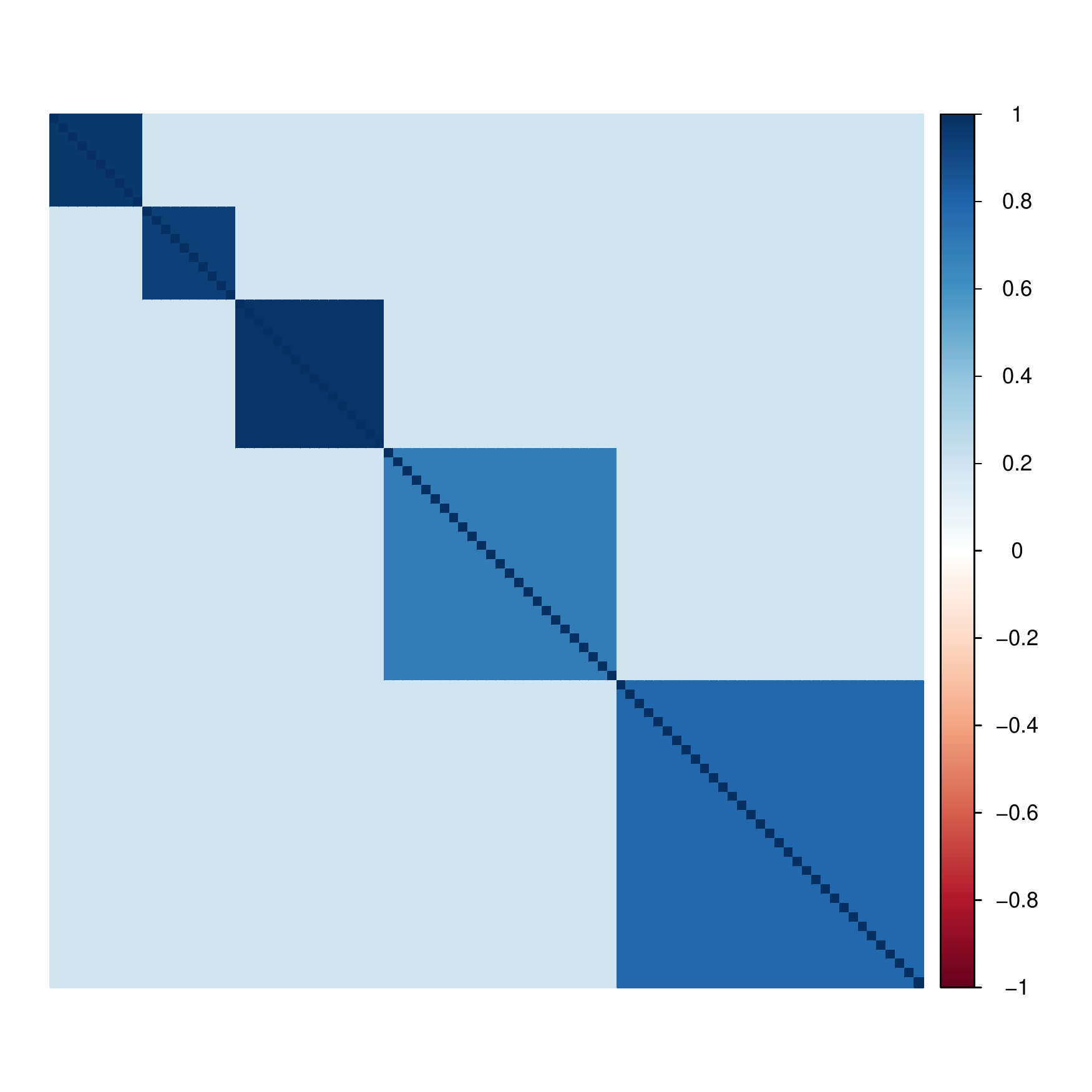}
\caption*{5 groups (a)}
\label{fig:res1ba}
\end{subfigure}%
\hspace*{0.05\linewidth}
\begin{subfigure}[h]{0.4\textwidth}
\centering 
\includegraphics[trim = 0cm 2cm 0cm 0cm, width=\textwidth]{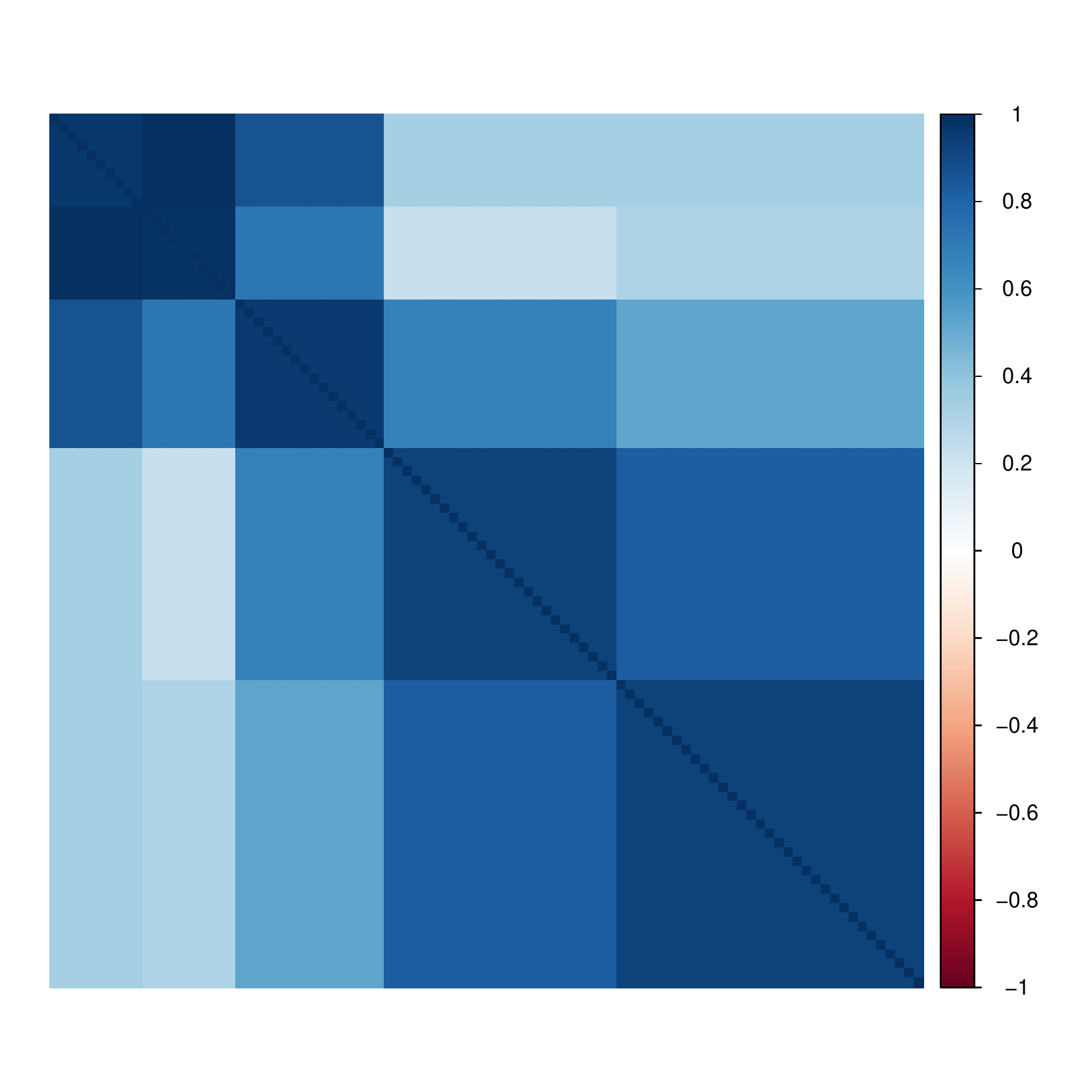}
\caption*{5 groups (b)}
\label{fig:groupe_gene1}
\end{subfigure}%
\caption{Estimated correlation kernel for the chemical element, 
with a common between-group covariance parameter (left) or different ones (right).}
\label{fig:res1b}
\end{figure}%

\begin{figure}[h] 
\centering
\begin{subfigure}[h]{0.4\textwidth}
\centering
\includegraphics[trim = 0cm 2cm 0cm 0cm, width=0.8\textwidth]{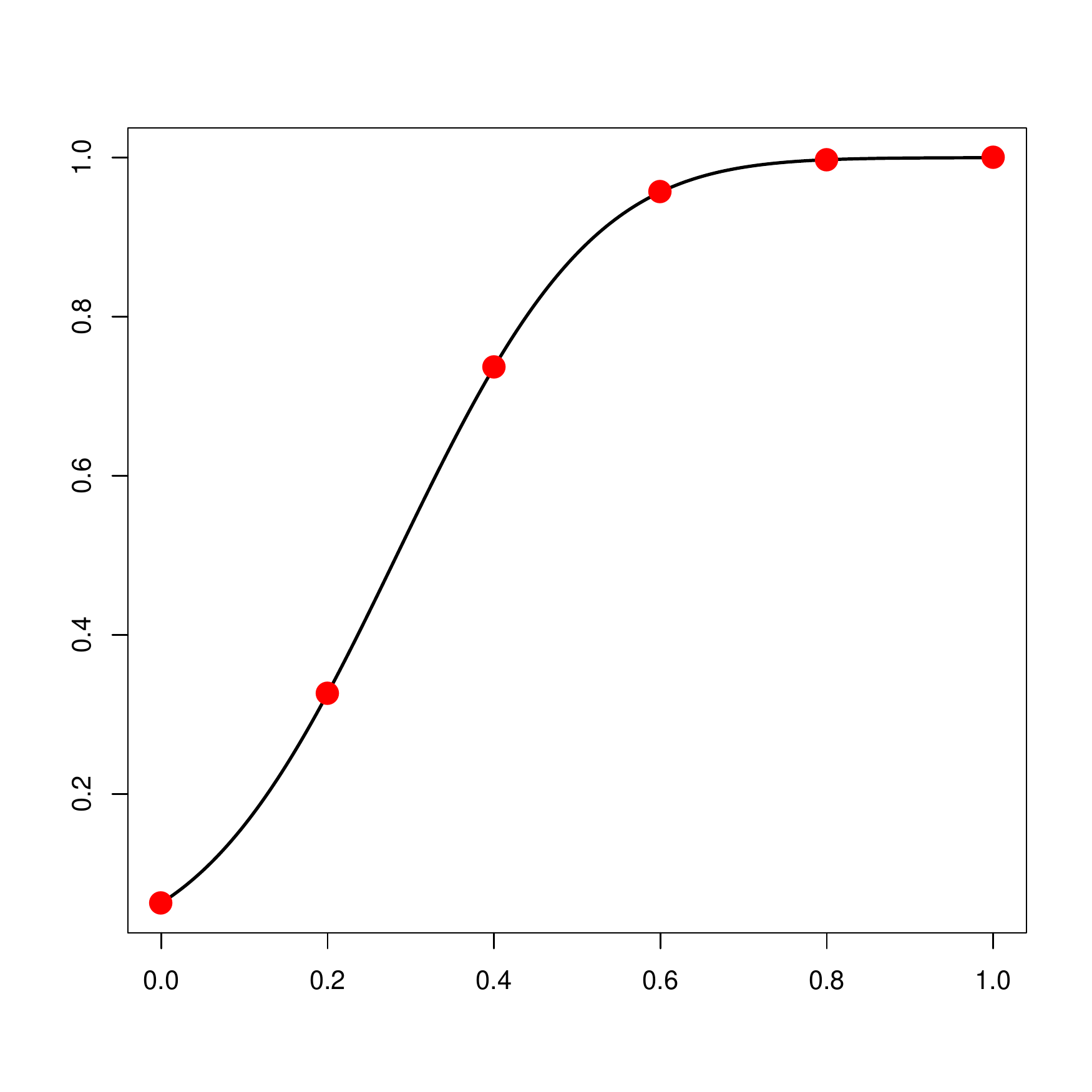}
\label{fig:res2a}
\end{subfigure}%
\hspace*{0.05\linewidth}
\begin{subfigure}[h]{0.4\textwidth}
\centering
\includegraphics[trim = 0cm 2cm 0cm 0cm, width=0.8\textwidth]{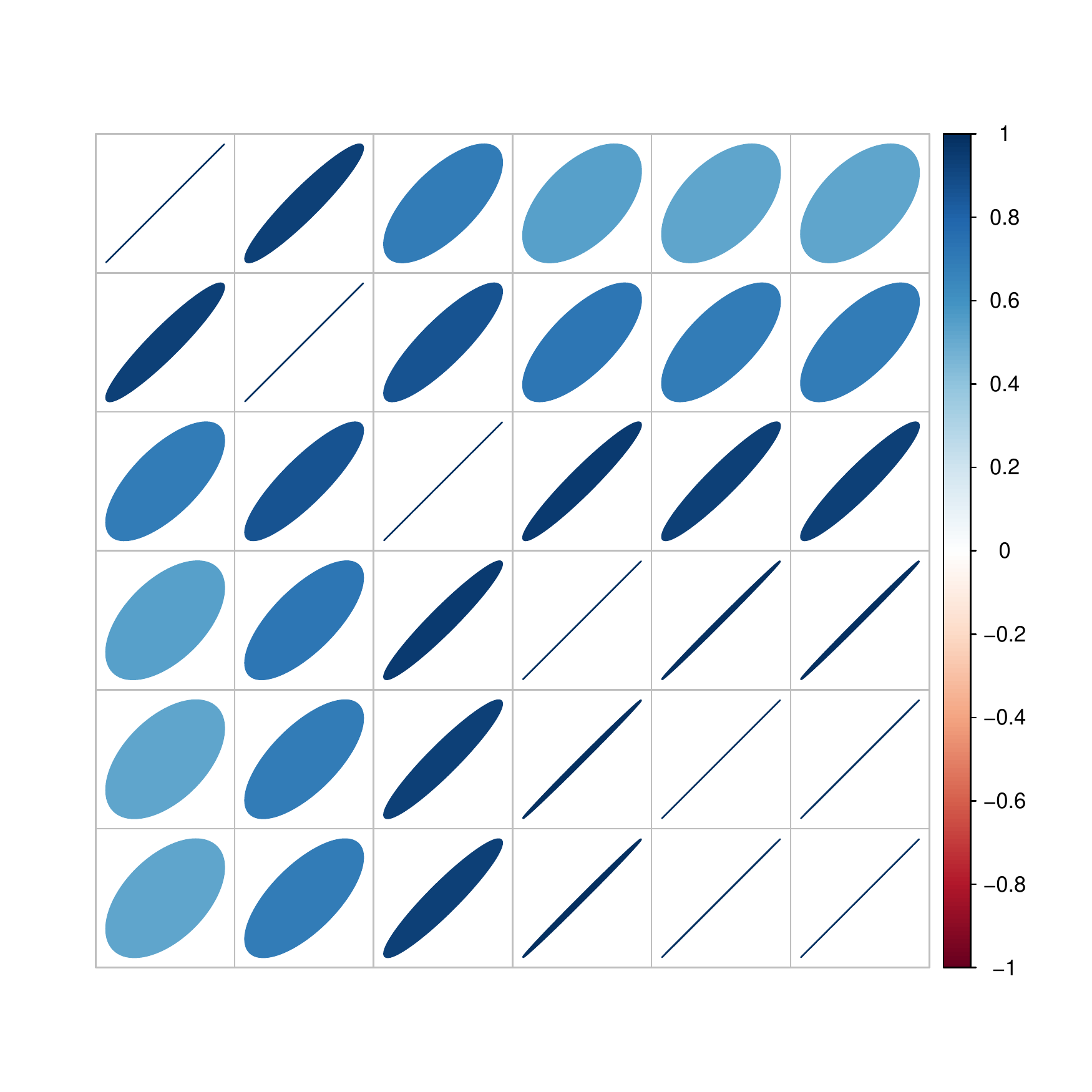}
\label{fig:res2b}
\end{subfigure}%
\caption{Estimated correlation kernel for the energy: 
estimated warping $F$ (left) and correlation structure (right).}
\label{fig:res2}
\end{figure}

\noindent
\begin{figure}[!h] 
\centering
\includegraphics[trim = 0cm 2cm 0cm 0cm, width=0.3\textwidth]{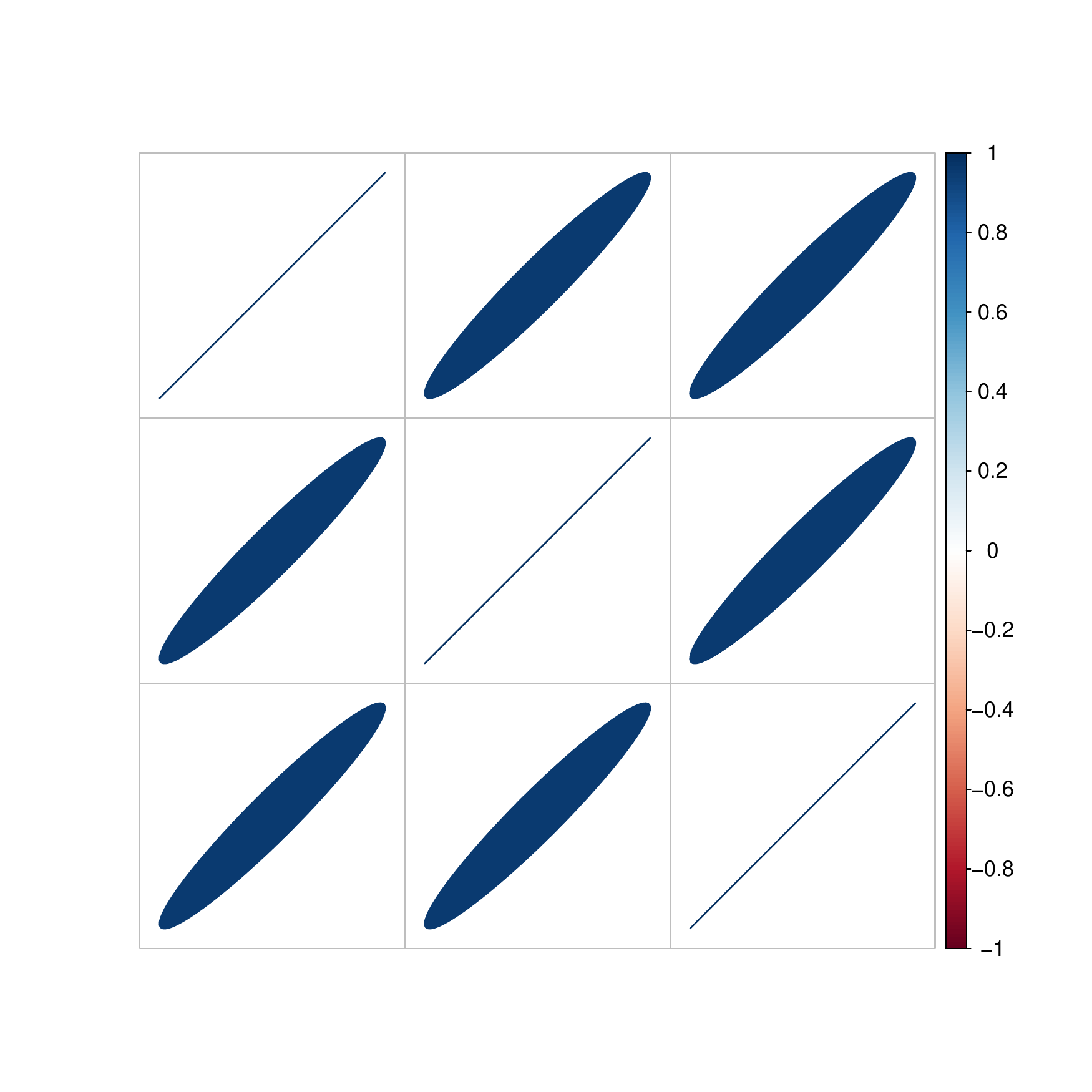}
\caption{Estimated correlation kernel for the geometric shape.}
\label{fig:resGS}
\end{figure}

%% file: presentation_appli.tex

\newcommand{\g}[1]{\boldsymbol{#1}}
\newcommand{\bleu}[1]{\color{blue}#1\color{black}}

As presented in Introduction, this research is originally motivated by the solving of an inverse problem confronting experimental measurements in nuclear engineering and time-consuming numerical simulation. 
More precisely, this analysis concerns the identification of the mass $m$ of $^{239}\textrm{Pu}$ that is present in a particular waste container using a non-destructive nuclear detection technique such as the gamma spectrometry \citep{knoll}. In that case, at each energy level $E$,

\begin{equation}
m \times \epsilon(E;\mathcal{E}) = y^{\text{SG}}(E), \label{eq1} \end{equation}

\noindent{}where $y^{\text{SG}}(E)$ is the quantity of interest provided by the gamma transmitter, 
and $\epsilon(E;\mathcal{E})$ is the attenuation coefficient, which depends on the source environment denoted by $\mathcal{E}$. In practice, only discrete values of $E$ are of interest,
corresponding to the natural energy levels of $^{239}\textrm{Pu}$:

\begin{equation}
E \ \in \ \{94.66,129.3,203.6,345.0,375.1,413.7\} \ \text{(keV)}.
\end{equation}

Then, based on previous studies \citep{guillot}, the real source environment 
is parameterized by the following input variables:

\begin{itemize}
\item An equivalent geometric shape for 
the nuclear waste: sphere (`sph'), cylinder (`cyl') or parallelepiped (`par'). 
\item An equivalent material for this waste, characterized by its chemical element with 
atomic number in \{1, \dots, 94\}, 
\item The bulk density of the waste, in $[0,1]$, 
\item The distance of measurement between the container and the measurement device, in $[80, 140]$ (cm), 
\item The mean width and lateral surfaces (in logarithmic scale) crossed by a gamma ray during the rotation of the object. 
\end{itemize}

After normalization, the characteristics of the input space can be summed up in Table \ref{tabIS}.

\begin{table}[!h]
\begin{center}
\begin{tabular}{ccc} \hline 
Name of the input & & Variation domain \\ \hline  
Distance  & & [0,1] \\
Density & & [0,1] \\
Width & & [0,1] \\
Surface & & [0,1] \\
Energy & & $\{ 1,2,3,4,5,6\}$ \\
Shape  & & \{sph, cyl, par\} \\
Chemical element & & $\{ 1,\dots,94\}$ \\ \hline
\end{tabular}
\caption{Description of the input variables for the nuclear application.}
\label{tabIS}
\end{center}
\end{table}

To recapture the notation of the previous sections, 
let $\mb{x}$ and $\mb{u}$ be the vectors gathering respectively the continuous and categorical inputs,
and $\mb{w} = (\mb{x}, \mb{u})$. 
For a given value of $\mb{w}$, Monte Carlo simulation codes as MCNP \citep{mcnp} can be used to model the measured scene and approach the value of 
$\epsilon(\mb{w}) = \epsilon(E,\mathcal{E})$. 
The mass $m$ can eventually be searched as the solution of the following optimization problem:

\begin{equation} \label{eq:OptProblem}
(m^\star, \mb{w}^\star) = \arg\min_{m, \mb{w}} \Vert \g{y}^{\text{obs}} - m \times \g{\epsilon}(\mb{w}) \Vert, \end{equation}

\noindent{}where $\Vert \cdot \Vert$ is the classical Euclidian norm, 
$\g{\epsilon}(\mb{w})$ and $\g{y}^{\text{obs}}$ respectively gather the values of $\epsilon$ and $y^{\text{SG}}$
at the six values of $E$ that are used for the measurements. 
To solve~(\ref{eq:OptProblem}), it is therefore necessary to compute $\g{\epsilon}$ at a high number of points. 
However, 
each evaluation of the MCNP code can be extremely demanding (between several minutes to several hours CPU for one evaluation). 
Thus, surrogate models have to be introduced to emulate the function $\mb{w} \mapsto \epsilon(\mb{w})$,
which is now investigated in the frame of Gaussian process regression.
We refer to \cite{Clement_et_al_2018} for the second step, namely the treatment of the inversion problem.

%% file: conclusion.tex

In the framework of GP regression with both continuous and categorical inputs,
we focus on problems where categorical inputs may have a potentially large number of levels $\level$,
partitioned in $\nbGroup \ll \level$ groups of various sizes. 
We provide new results about parsimonious block covariance matrices,
defined by a few within- and between-group covariance parameters.\\

We revisit a two-way nested Bayesian linear model, 
where the response term is defined as a sum of a group effect and a level effect.
We obtain a flexible parameterization of block covariance matrices
which automatically satisfy the positive definiteness conditions.
As a particular case, we recover situations where the within-group covariance structures are compound symmetry,
with possible negative correlations.
Furthermore, we show that the positive definiteness of a given block covariance matrix 
can be checked by verifying that the small matrix of size $\nbGroup$
obtained by averaging each block is positive definite.
This criterion can be useful if the proposed block matrix has a desirable constraint, 
such as homoscedasticity, which is not directly handled by the proposed parameterization.\\

We apply these findings on several toy functions as well as an application in nuclear engineering,
with 4 continuous inputs, 3 categorical inputs, one of them having 94 levels 
corresponding to chemical numbers in Mendeleev's table.
In this application, 5 groups were defined by experts.
The results, measured in terms of prediction accuracy, outperform
those obtained with simpler assumptions, such as gathering all levels into a single group.
It is gratifying that our nominal method performs almost as well as 
using the appropriate order with a warped kernel.\\

There are several perspectives for this work.
First, one future direction is to find a data-driven technique to recover groups of levels,
made more difficult when a small number of observations is available.
Similarly, if there is an order between levels, can we infer it from the data?
Second, the trend of the GP models with mixed continuous and categorical inputs could be made more complex,
in the same vein as works on GP models with continuous inputs. 

%% file: appendix.tex

\begin{proof}[Proof of Proposition~\ref{prop:CSparam}]
The vector $(\bm{\levelEffect}, \levelEffect_1 + \dots + \levelEffect_\level)$ is a centered Gaussian vector with covariance matrix
$$ \varIndiv  \begin{pmatrix}
\I{\level} & \one{\level} \\
\one{\level}^\top & \level
\end{pmatrix}.
$$
Hence the conditional distribution of $\bm{\levelEffect}$ knowing $\overline{\levelEffect} = 0$
is a centered Gaussian vector with covariance matrix
$$ \Cov(\bm{\levelEffect} \, \vert \,  \overline{\levelEffect} = 0 ) 
= \varIndiv [\I{\level} - \one{\level} \level^{-1} \one{\level}^\top] 
= \varIndiv [\I{\level} - \level^{-1} \J{\level}] .$$
Then, by using the independence between $\groupEffect$ and the $\levelEffect_\levelIndex$'s, we deduce
\begin{eqnarray*}
\Cov(\bm{\responseAnova} \, \vert \, \overline{\levelEffect} = 0 ) 
&=& \varCentre \J{\level} + \varIndiv [\I{\level} - \level^{-1} \J{\level}] \\
&=& \varIndiv \I{\level} + [\varCentre - \level^{-1} \varIndiv] \J{\level}.
\end{eqnarray*} 
We recognize the CS covariance matrix $\compound{\level} {v, c}$ with 
$ v = \varCentre + (1-\level^{-1}) \varIndiv$ 
and $c = \varCentre - \level^{-1} \varIndiv$.
As a covariance matrix, it is positive semidefinite. 
Furthemore, we have $c < v$ and $c + (\level-1)^{-1}v = \varCentre [1 + (\level-1)^{-1}] > 0$,
and the conditions of positive definiteness (\ref{eq:CSpdCondition}) are satisfied.\\
Conversely, let $\mb{C}$ be a positive definite CS matrix $\compound{\level}{v, c}$.
Then we have $ -(\level-1)^{-1} v < c < v$, 
and we can define $ \varCentre = \level^{-1}[v + (\level-1)c]$ and $\varIndiv = v - c$.
From the direct sense, we then obtain that the covariance matrix 
of $\bm{\responseAnova} \, \vert \, \overline{\levelEffect} = 0 $ 
is $\compound{\level} {v, c} = \mb{C}$.
\end{proof}

\begin{proof}[Proof of Proposition~\ref{prop:covDegenerate}]
The first part of the proposition is obtained by remarking that if $\varIndivG{} = \Cov(\mb{z})$, 
then $\varIndivGbar{} = \Var(\overline{z})$. 
Thus, assuming that $\mb{z}$ is centered, $\varIndivGbar{}=0$ is equivalent to $\overline{z}=0$ with probability $1$.\\
For the second part, notice that $\overline{z} = 0$ means that $\mb{z}$ is orthogonal to $\one{\level}$. 
Thus, one can write the expansion of $\mb{z}$ in the orthonormal basis $\one{\level}^\perp$ defined by $\constrastMat{}$.
Denoting by $\mb{t}$ the $(\level-1)$-vector of coordinates, we have $\mb{z} = \constrastMat{} \mb{t}$.
This gives $\varIndivG{} = \Cov(\constrastMat{} \mb{t}) = \constrastMat{} \, \Cov(\mb{t}) \constrastMat{}^\top$, and (\ref{eq:FgParam}) follows with $\varZeroAverageRed{} = \Cov(\mb{t})$.\\
To prove unicity, observe that, by definition, $\constrastMat{}^\top \constrastMat{} = \I{\level-1}, \constrastMat{}^\top \one{\level} = 0$.
Starting from $\varIndivG{} = \constrastMat{} \varZeroAverageRed{} \constrastMat{}^\top$, and multiplying by $\constrastMat{}^\top$ on the left and by $\constrastMat{}$ on the right, 
we get $\varZeroAverageRed{} = \constrastMat{}^\top \varIndivG{} \constrastMat{}$, showing that $\varZeroAverageRed{}$ is unique.\\
Now, let $\varIndivG{} = v [\I{\level} - \level^{-1} \J{\level}]$. Since $\J{\level} = \one{\level}\one{\level}^\top$, we obtain 
$$\varZeroAverageRed{} = \constrastMat{}^\top \varIndivG{} \constrastMat{} = v [\constrastMat{}^\top \constrastMat{} - \level^{-1} (\constrastMat{}^\top \one{\level})(\one{\level}^\top \constrastMat{})] = v \I{\level-1}.$$
As a by-product, notice that resubstituting $\varZeroAverageRed{}$ into 
$\varIndivG{} = \constrastMat{} \varZeroAverageRed{} \constrastMat{}^\top$ 
gives ${\constrastMat{} \constrastMat{}^\top = \I{\level} - \level^{-1} \J{\level}}$.\\
Finally, if $\mb{z} \sim \mathcal{N}(0, v \I{\level})$, then the properties of conditional Gaussian vectors 
lead immediately to $\Cov(\mb{z} \vert \overline{z}=0) = \varIndivG{}$.
\end{proof}

\begin{proof}[Proof of Theorem~\ref{prop:blockParam}]
The expressions of $\mb{\within}_g$ and $\between_{g,g'}$ are obtained directly 
by using the independence assumptions about $\bm{\groupEffect}$ and the $\bm{\levelEffect}$'s. 
Notice that $\varIndivG{g}$, the covariance matrix of  $\bm{\levelEffect}_{g / .}$ knowing $\overline{\levelEffect_{g / .}} = 0$, is centered by Proposition~\ref{prop:covDegenerate}. 
This gives $\mb{\within}_g - \overline{\within_g} \J{n_g} = \varIndivG{g}$, which is positive semidefinite.
Hence $\mb{\matCovCat}$ is a GCS block matrix.

Conversely, let $\mb{\matCovCat}$ be a positive semidefinite GCS block matrix.
Let $\widetilde{\mb{\matCovCat}}$ be the matrix obtained from $\mb{\matCovCat}$ by averaging each block.
Then $\widetilde{\mb{\matCovCat}}$ is also a positive semidefinite matrix.
Indeed, since $\mb{\matCovCat}$ is positive semidefinite, it is the covariance matrix of some vector $\mb{z}$.
Then $\widetilde{\mb{\matCovCat}}$ is the covariance matrix of $\widetilde{\mb{z}}$, the vector obtained from $\mb{z}$ by averaging by group:
$\widetilde{z}_g =  n_g^{-1} \sum_{\levelIndex \in G_g} z_\levelIndex$.
Thus there exists a centered Gaussian vector $(\groupEffect_g)_{1 \leq g \leq p}$ whose covariance matrix is 
$$\varCentreGroup = \widetilde{\mb{\matCovCat}}.$$
Now, for $g = 1, \dots, \nbGroup$, define 
$$\varIndivG{g} = \mb{\within}_g - \varCentreGroupCoef{g,g} \J{n_g} = \mb{\within}_g - \overline{\within_g} \J{n_g}.$$
Observe that $\varIndivGbar{g} = 0$, and by assumption $\varIndivG{g}$ is positive semidefinite.
Hence, from Proposition~\ref{prop:covDegenerate},
there exists a centered Gaussian vector $(\levelEffect_{g/j})_{1 \leq \levelIndex \leq n_g}$ such that 
$$\varIndivG{g} = \Cov(\bm{\levelEffect}_{g/.} \vert \overline{\levelEffect_{g / .}} = 0).$$
We can assume that $\bm{\levelEffect}_{1 / .}, \dots, \bm{\levelEffect}_{\nbGroup / .}$ are independent,
and $\bm{\groupEffect}$ and $\bm{\levelEffect}$ are independent.
Finally, we set $\responseAnova_{g/\levelIndex} = \groupEffect_g + \levelEffect_{g/\levelIndex}$. 
By the direct sense and (\ref{eq:blockCov}), we obtain that $\mb{\matCovCat}$ is the covariance matrix of 
$\bm{\responseAnova}$ conditional on $\{ \overline{\levelEffect_{g / .}} = 0, \, g=1, \dots, \nbGroup\}.$
\end{proof}

\begin{proof}[Proof of Corollary~\ref{prop:blockCSparam}] 
Let $\mb{\matCovCat}$ be a positive semidefinite GCS block matrix 
with CS diagonal blocks. 
Then the diagonal CS matrices are positive semidefinite, leading to $v_g - c_g \geq 0$.
Thus,
$$\mb{\within}_g - \overline{\within_g} \J{n_g} = (v_g - c_g) (\I{n_g} - n_g^{-1} \J{n_g})$$
is a positive semidefinite CS matrix. 
Hence, by Theorem~\ref{prop:blockParam}, $\mb{\matCovCat}$ is obtained from Model~(\ref{eq:blockMatModel}),
with $\Cov( \bm{\levelEffect}_{g / .} \vert \overline{\levelEffect_{g / .}} = 0) = (v_g - c_g) (\I{n_g} - n_g^{-1} \J{n_g})$.
By Prop.~\ref{prop:covDegenerate} (last part), we can choose $\varIndivG{g} = v_{\levelEffect_g} \I{n_g}$,
with $v_{\levelEffect_g} = v_g - c_g \geq 0$.\\
Conversely, if $\varIndivG{g} = v_{\levelEffect, g} \I{n_g}$, then by Prop.~\ref{prop:covDegenerate}, 
$\varIndivG{g} = \Cov(\bm{\levelEffect}_{g / .} \vert \overline{\levelEffect_{g / .}} = 0)$ 
is a CS covariance matrix.
The result follows by Theorem~\ref{prop:blockParam}.
\end{proof}

\begin{proof}[Proof of Theorem~\ref{prop:TandTbar}] 
First observe that Eq.~(\ref{eq:TandTbarEq}) is straightforward.\\
The direct sense of (i) has already been derived in the proof of Prop.~\ref{prop:blockParam}:
 $\widetilde{\mb{\matCovCat}}$ is the covariance matrix of $\widetilde{\mb{z}}$.\\
Conversely, inspecting that proof, we see that if $\widetilde{\mb{\matCovCat}}$ is positive semidefinite, 
then $\mb{\matCovCat}$ admits the representation~(\ref{eq:blockMatModel}).
Thus $\mb{\matCovCat}$ is a covariance matrix, and positive semidefinite.
This can be also proved from Eq.~(\ref{eq:TandTbarEq}):
the two terms of the right-hand side are positive semidefinite, thus so is their sum.\\

Now consider (ii). If $\mb{\matCovCat}$ is positive definite, 
then its diagonal blocks $\mb{\within}_g$ are all positive definite. 
Furthermore, by (i), $\widetilde{\mb{\matCovCat}}$ is positive semidefinite.
It if were singular, there would exist a non-zero vector $\bm {\gamma}$ 
such that $\bm{\gamma}^\top \widetilde{\mb{z}} =0$ with probability $1$.
This gives a non-trivial linear combination of $\mb{z}$ which is equal to zero,
and $\mb{\matCovCat}$ would not be positive definite. 
Thus $\widetilde{\mb{\matCovCat}}$ is positive definite.

Conversely, let us assume that $\widetilde{\mb{\matCovCat}}$ and all $\mb{\within}_g$'s are positive definite.
We will need the following Lemma:
\newtheorem{lemma}{Lemma}
\begin{lemma}
Let $\mb{F}$ be a centered covariance matrix of size $n$, with rank $n-1$. 
If for some vector $\bm{\beta}$, we have $\bm{\beta}^\top \mb{F} \bm{\beta}=0$, 
then $\bm{\beta}$ is a constant vector.
\end{lemma}
\begin{proof}
A symmetric square root $\mb{L}$ with $\mb{L}^2 = \mb{F}$ has also rank $n-1$, 
hence the nullspace of $\mb{L}$ is one-dimensional.
Since $\mb{F}$ is centered, we have $\one{n}^\top \mb{F} \one{n}= 0$.
Thus $\mb{L} \one{n}=0$.
Similarly, $\mb{L}\bm{\beta}=0$. The result follows.
\end{proof}

\noindent Now, let $\bm{\beta}$ be a vector of length $\sum_g  n_g$ such that 
$\bm{\beta}^\top \mb{\matCovCat} \bm{\beta} = 0$. 
Eq. (\ref{eq:TandTbarEq}) gives
$$\bm{\beta}^\top \mb{\matCovCat} \bm{\beta} =:  U + V$$
where $U$ and $V$ are obtained from each term in (\ref{eq:TandTbarEq})
by left-multiplying by $\bm{\beta}^\top$ and right-multiplying by $\bm{\beta}$.
Since $U$ and $V$ are non-negative, we must have $U=V=0$.
\begin{itemize}
\item $V=0$ implies that all subvectors $\bm{\beta}_g$ corresponding to groups $g$ verify
$\bm{\beta}_g^\top [\mb{\within}_g - \overline{\within_g} \J{n_g}] \bm{\beta}_g = 0$.
This implies that they are all constant vectors.
Indeed, since (by assumption) $\mb{\within}_g$ has rang $n_g$ and $\overline{\within_g} \J{n_g}$ has rank 1, 
the centered matrix $\mb{F}_g := \mb{\within}_g - \overline{\within_g} \J{n_g}$ must have rank $n_g - 1$,
and the result is given by the lemma.
\item $U=0$ gives $\mb{X}^\top \bm{\beta}=0$ by positive definiteness of $\widetilde{\mb{\matCovCat}}$.
Now, $\mb{X}^\top \bm{\beta}$ is a vector of length $G$ whose component $g$
is the sum of $\bm{\beta}_g$ coefficients. 
\end{itemize}
Gathering the conclusions of the two items gives $\bm{\beta} = 0$.
Finally $\mb{T}$ is positive definite.

\paragraph{Remark.} 
Notice that for (ii) we needed to add the condition that $\mb{\within}_g$ is positive definite for all $g=1, \dots, \nbGroup$.
However, adding an equivalent condition for (i), namely that $\mb{\within}_g$ is positive semidefinite, 
was not necessary. Indeed, it is a consequence of the fact that $\mb{\within}_g - \overline{\within_g} \J{n_g}$ is positive semidefinite and that $\widetilde{\mb{\matCovCat}}$ is positive semidefinite, which implies 
$\widetilde{\matCovCat}_{g,g} = \overline{\within_g} \geq 0$.
\end{proof}

\begin{proof}[Proof of Proposition~\ref{prop:exclusionProperty}.]
By assumption, $\widetilde{\matCovCat}_{g,g} = \overline{\within_g} = 0$.
Now by Theorem~\ref{prop:TandTbar}, $\widetilde{\mb{\matCovCat}}$ is positive semidefinite.
As its diagonal term $\widetilde{\matCovCat}_{g,g}$ is zero, 
it implies that all the terms on the same row are zero.
Hence  for all $g' \neq g: \, \widetilde{\matCovCat}_{g,g'} = 0$. 
As $\mb{\matCovCat}$ is a GCS covariance matrix, its off-diagonal blocks are constant,
and thus, $T_{g,g'} = \widetilde{T}_{g,g'} = 0$, which proves that $\mb{y}_g$ and $\mb{y}_{-g}$ are non correlated.
The result follows by Gaussianity of $\mb{y}$.
\end{proof}

\begin{proof}[Proof of the assertion of Section~\ref{sec:summary}.] 
We claim that $\mb{\matCovCat}$ is invertible iff $\varCentreGroup$ and all the
 $\varZeroAverageRed{g}$'s are invertible.
Indeed, from Theorem~\ref{prop:TandTbar}, $\mb{\matCovCat}$ is invertible iff 
$\widetilde{\mb{\matCovCat}}$ and all the $\mb{\within}_g$'s are invertible.
Now, from Theorem~\ref{prop:blockParam}, $\widetilde{\mb{\matCovCat}} = \varCentreGroup$,
and 
$$\mb{\within}_g = \left[ \varCentreGroup \right]_{g,g} \J{n_g} 
+ \mb{A}_g \mb{\varZeroAverageRed}_{g} \mb{A}_g^\top,$$
where $\mb{A}_g$ is a $n_g$ by $n_g-1$ matrix whose colums vectors are an orthonormal basis of $\one{n_g}^\perp$. 
Left multiplying by $\mb{A}_g^\top$ and right multiplying by $\mb{A}_g$, 
and remarking that $\J{n_g}=\one{n_g}\one{n_g}^\top$, we obtain:
$ \mb{A}_g^\top \mb{\within}_g \mb{A}_g = \mb{\varZeroAverageRed}_{g}$.
It is now clear that $\mb{\within}_g$ is invertible 
iff $\mb{\varZeroAverageRed}_{g}$ is.
\end{proof}

%% file: groupKernels.bbl
\begin{thebibliography}{25}
\providecommand{\natexlab}[1]{#1}
\providecommand{\url}[1]{\texttt{#1}}
\expandafter\ifx\csname urlstyle\endcsname\relax
  \providecommand{\doi}[1]{doi: #1}\else
  \providecommand{\doi}{doi: \begingroup \urlstyle{rm}\Url}\fi

\bibitem[Chevalier et~al.(2014)Chevalier, Bect, Ginsbourger, Vazquez, Picheny,
  and Richet]{Chevalier_2014}
C.~Chevalier, J.~Bect, D.~Ginsbourger, E.~Vazquez, V.~Picheny, and Y.~Richet.
\newblock Fast parallel kriging-based stepwise uncertainty reduction with
  application to the identification of an excursion set.
\newblock \emph{Technometrics}, 56\penalty0 (4):\penalty0 455--465, 2014.

\bibitem[Clement et~al.(2018)Clement, Saurel, and Perrin]{Clement_et_al_2018}
A.~Clement, N.~Saurel, and G.~Perrin.
\newblock Stochastic approach for radionuclides quantification.
\newblock \emph{EPJ Web Conf.}, 170:\penalty0 06002, 2018.
\newblock URL \url{https://doi.org/10.1051/epjconf/201817006002}.

\bibitem[Deng et~al.(2017)Deng, Lin, Liu, and Rowe]{Deng2017}
X.~Deng, C.~D. Lin, K.-W. Liu, and R.~K. Rowe.
\newblock Additive {G}aussian process for computer models with qualitative and
  quantitative factors.
\newblock \emph{Technometrics}, 59\penalty0 (3):\penalty0 283--292, 2017.

\bibitem[Deville et~al.(2015)Deville, Ginsbourger, and Roustant]{kergp}
Y.~Deville, D.~Ginsbourger, and O.~Roustant.
\newblock \emph{kergp: {G}aussian Process Laboratory}, 2015.
\newblock URL \url{https://CRAN.R-project.org/package=kergp}.
\newblock Contributors: N. Durrande. R package version 0.2.0.

\bibitem[Fox and Dunson(2012)]{multiresolutionGP_Fox_Dunson}
E.~Fox and D.~B. Dunson.
\newblock Multiresolution {G}aussian processes.
\newblock In F.~Pereira, C.~J.~C. Burges, L.~Bottou, and K.~Q. Weinberger,
  editors, \emph{Advances in Neural Information Processing Systems 25}, pages
  737--745. Curran Associates, Inc., 2012.
\newblock URL
  \url{http://papers.nips.cc/paper/4682-multiresolution-{G}aussian-processes.pdf}.

\bibitem[Goorley et~al.(2013)Goorley, Fensin, and McKinney]{mcnp}
J.~T. Goorley, M.~Fensin, and G.~McKinney.
\newblock \emph{MCNP6 User’s Manual, Version 1.0}, May 2013.

\bibitem[Gower(1982)]{gower1982}
J.~C. Gower.
\newblock Euclidean distance geometry.
\newblock \emph{Math. Sci}, 7\penalty0 (1):\penalty0 1--14, 1982.

\bibitem[Guillot(2015)]{guillot}
N.~Guillot.
\newblock \emph{Quantification gamma de radionucl\'{e}ides par mod\'{e}lisation
  \'{e}quivalente}.
\newblock PhD thesis, {Universit\'{e} Blaise Pascal - Clermont-Ferrand II},
  France, 2015.

\bibitem[Khuri and Good(1989)]{Khuri1989}
A.~Khuri and I.~Good.
\newblock The parameterization of orthogonal matrices : a review mainly for
  statisticians : review paper.
\newblock \emph{South African Statistical Journal}, 23\penalty0 (2):\penalty0
  231--250, 1989.

\bibitem[Knoll(2010)]{knoll}
G.~F. Knoll.
\newblock \emph{Germanium Gamma-Ray Detectors}, volume~3.
\newblock John Wiley \& Sons, 2010.

\bibitem[Lindley and Smith(1972)]{Lindley_Smith_1972}
D.~Lindley and A.~Smith.
\newblock Bayes estimate for the linear model (with discussion) part 1.
\newblock \emph{Journal of the Royal Statistical Society, Ser B}, 34\penalty0
  (1):\penalty0 1--41, 1972.

\bibitem[McCullagh(1980)]{mccullagh1980}
P.~McCullagh.
\newblock Regression models for ordinal data.
\newblock \emph{Journal of the Royal Statistical Society. Series B
  (Methodological)}, 42\penalty0 (2):\penalty0 109--142, 1980.

\bibitem[Park and Choi(2010)]{Park_Shoi_2010}
S.~Park and S.~Choi.
\newblock Hierarchical {G}aussian process regression.
\newblock In M.~Sugiyama and Q.~Yang, editors, \emph{Proceedings of 2nd Asian
  Conference on Machine Learning}, volume~13 of \emph{Proceedings of Machine
  Learning Research}, pages 95--110, 2010.

\bibitem[Pinheiro and Bates(2009)]{pinheiro2009}
J.~Pinheiro and D.~Bates.
\newblock \emph{Mixed-Effects Models in S and S-PLUS}.
\newblock Statistics and Computing. Springer New York, 2009.

\bibitem[Pinheiro and Bates(1996)]{Pinheiro1996}
J.~C. Pinheiro and D.~M. Bates.
\newblock Unconstrained parametrizations for variance-covariance matrices.
\newblock \emph{Statistics and Computing}, 6\penalty0 (3):\penalty0 289--296,
  1996.

\bibitem[Qian(2012)]{qian2012}
P.~Z.~G. Qian.
\newblock Sliced {L}atin hypercube designs.
\newblock \emph{Journal of the American Statistical Association}, 107\penalty0
  (497):\penalty0 393--399, 2012.

\bibitem[Qian et~al.(2007)Qian, Wu, and Wu]{qian2007}
P.~Z.~G. Qian, H.~C.~F. Wu, and J.~Wu.
\newblock {G}aussian process models for computer experiments with qualitative
  and quantitative factors.
\newblock Technical report, Department of statistics, University of Wisconsin,
  2007.

\bibitem[Rasmussen and Williams(2006)]{rasmussen2005}
C.~Rasmussen and C.~Williams.
\newblock \emph{{G}aussian Processes for Machine Learning}.
\newblock The MIT Press, 2006.

\bibitem[Sacks et~al.(1989)Sacks, Welch, Mitchell, and Wynn]{Sacks_1989}
J.~Sacks, W.~Welch, T.~Mitchell, and H.~Wynn.
\newblock Design and analysis of computer experiments.
\newblock \emph{Statistical Science}, \penalty0 (4):\penalty0 409--435, 1989.

\bibitem[Shepard et~al.(2015)Shepard, Brozell, and Gidofalvi]{Ron2015}
R.~Shepard, S.~R. Brozell, and G.~Gidofalvi.
\newblock The representation and parametrization of orthogonal matrices.
\newblock \emph{The Journal of Physical Chemistry A}, 119\penalty0
  (28):\penalty0 7924--7939, 2015.

\bibitem[Smith(1973)]{smith1973}
A.~Smith.
\newblock Bayes estimates in one-way and two-way models.
\newblock \emph{Biometrika}, 60\penalty0 (2):\penalty0 319--329, 1973.

\bibitem[Venables and Ripley(2002)]{Venable_Ripley_MASS}
W.~N. Venables and B.~D. Ripley.
\newblock \emph{Modern Applied Statistics with S}.
\newblock Springer, 4 edition, 2002.

\bibitem[Wei and Simko(2016)]{corrplot}
T.~Wei and V.~Simko.
\newblock \emph{corrplot: Visualization of a Correlation Matrix}, 2016.
\newblock R package version 0.77.

\bibitem[Wickham(2009)]{ggplot2}
H.~Wickham.
\newblock \emph{ggplot2: Elegant Graphics for Data Analysis}.
\newblock Springer-Verlag New York, 2009.
\newblock ISBN 978-0-387-98140-6.
\newblock URL \url{http://ggplot2.org}.

\bibitem[Zhang and Notz(2015)]{zhang2015}
Y.~Zhang and W.~I. Notz.
\newblock Computer experiments with qualitative and quantitative variables: A
  review and reexamination.
\newblock \emph{Quality Engineering}, 27\penalty0 (1):\penalty0 2--13, 2015.

\end{thebibliography}
